\documentclass[a4paper,oneside,10pt]{article}%
\usepackage{amsfonts}
\usepackage{amsmath,amssymb}
\usepackage{amsmath}
\usepackage{amssymb}
\usepackage{graphicx}
\usepackage{hyperref}
\usepackage{color}%
\providecommand{\U}[1]{\protect\rule{.1in}{.1in}}
\newtheorem{theorem}{Theorem}[section]
\newtheorem{corollary}[theorem]{Corollary}
\newtheorem{definition}[theorem]{Definition}
\newtheorem{example}[theorem]{Example}
\newtheorem{lemma}[theorem]{Lemma}

\newtheorem{proposition}[theorem]{Proposition}
\newtheorem{remark}[theorem]{Remark}
\newenvironment{proof}[1][Proof]{\noindent \textbf{#1.} }{\  $\Box$}
\numberwithin{equation}{section}
\begin{document}

\title{Exit times for semimartingales under nonlinear expectation}
\author{Guomin Liu \thanks{School of Mathematical Sciences, Fudan University,
gmliu@fudan.edu.cn.} }
\maketitle

\textbf{Abstract}. Let $\mathbb{\hat{E}}$ be the upper expectation of a weakly
compact but possibly non-dominated family $\mathcal{P}$ of probability
measures. Assume that $Y$ is a $d$-dimensional $\mathcal{P}$-semimartingale
under $\mathbb{\hat{E}}$. Given an open set $Q\subset\mathbb{R}^{d}$, the exit
time of $Y$ from $Q$ is defined by
\[
{\tau}_{Q}:=\inf\{t\geq0:Y_{t}\in Q^{c}\}.
\]
The main objective of this paper is to study the quasi-continuity properties
of ${\tau}_{Q}$ under the nonlinear expectation $\mathbb{\hat{E}}$. Under some
additional assumptions on the growth and regularity of $Y$, we prove that
${\tau}_{Q}\wedge t$ is quasi-continuous if $Q$ satisfies the exterior ball
condition. We also give the characterization of quasi-continuous processes and
related properties on stopped processes. In particular, we obtain the
quasi-continuity of exit times for multi-dimensional $G$-martingales, which
nontrivially generalizes the previous one-dimensional result of Song
\cite{Song}.

{\textbf{Key words:} Nonlinear expectation,} $G$-expectation,
Multi-dimensional nonlinear semimartingales, Exit times, Quasi-continuity.

\textbf{AMS 2010 subject classifications:} 	60G40, 60G44, 60G48, 60H10 \addcontentsline{toc}{section}{\hspace*{1.8em}Abstract}

\section{Introduction}

On the space $\Omega$ of continuous paths, equipped with the topology of uniform convergence on compact sets and filtration generated by the canonical process, let $\mathcal{P}$ be a weakly
compact but possibly non-dominated family of probability measures. We define
the corresponding upper expectation and upper capacity by
\[
\mathbb{\hat{E}}[\xi]:=\sup_{P\in\mathcal{P}}E_{P}[\xi],\ c(A):=\sup
_{P\in\mathcal{P}}P(A),\ \ \ \text{for random variable}%
\ \xi\ \text{and measurable set}\ A.
\]
Assume that the  process $Y$ is a $d$-dimensional  (nonlinear) $\mathcal{P}%
$-semimartingale, i.e., $Y$ is a semimartingale under each $P\in\mathcal{{P}}$. A typical case of such kind of nonlinear expectation and nonlinear
 semimartingales is the notion of $G$-expectation and $G$-martingales proposed by Peng \cite{peng2005,Peng 1}. Given an open set $Q\subset\mathbb{R}^{d}$, we define the exit time of $Y$
from $Q$ by
\[
{\tau}_{Q}(\omega):=\inf\{t\geq0:Y_{t}(\omega)\in Q^{c}\}, \ \ \ \text{for}%
\ \omega\in\Omega.
\]
The aim of this paper is to study the quasi-continuity problem of exit times
${\tau}_{Q}$ under the nonlinear expectation $\mathbb{\hat{E}}$.

We say that a random variable is quasi-continuous, if it is continuous outside
an open set with any given small capacity, see Denis, Hu and Peng
\cite{D-H-P}. As is well-known, according to Lusin's theorem, all the
Borel measurable random variables defined on a Polish space are
quasi-continuous under the linear probability. This is the case that $\mathcal{P}$ is reduced to a single
measure. But it is no longer obvious for the general case since the elements
in the family $\mathcal{P}$ can be infinite, mutually singular and
non-dominated. Roughly speaking, the quasi-continuous random variables are
those that can be regarded as the limit of elements in $C_{b}(\Omega)$ in some proper sense, where $C_{b}(\Omega)$ is the set of bounded
continuous functions on $\Omega.$ Many important properties in the nonlinear
expectation theory, for example, monotone convergence theorem for decreasing
sequence (monotone convergence theorem for increasing sequence is trivial
since $\mathbb{\hat{E}}$ is an upper expectation)  and
(forward and backward) stochastic differential equations driven by
$G$-Brownian motion, only hold for random variables with such kind of regularity, see  \cite{D-H-P},   Gao \cite{Gao}, Hu, Ji, Peng and Song \cite{HJPS} and
Hu, Lin and Hima \cite{HLH}.

So one of the most important problems in the nonlinear expectation theory is
to verify that whether a random variable is quasi-continuous, especially for
stopping times in the form of ${\tau}_{Q}$ since such kind of problems keep
occurring when we stop a process as we often do in the classical analysis. The
first breakthrough on this direction was due to Song \cite[2011]%
{Song} (see also Song \cite[2014]{Song2}) who solved the quasi-continuity
problem of exit times when $Y$ is a one-dimensional $G$-martingale and
$Q=(-\infty,a)$. But the method of Song relies on a very important observation
that $Y_{{\tau}_{Q}\wedge t}\geq Y_{{\tau}_{\overline{Q}}\wedge t}$, which
holds only when $d=1$ and $Q=(-\infty,a)$, and hence cannot be applied to the
more general situation. So it remains a fascinating and challenging open
problem to establish the quasi-continuity of exit times for general dimension
$d$ and domain $Q$.

The main purpose of this paper is to provide a general theory on the
quasi-continuity properties of exit times ${\tau}_{Q}$, which allows us to
maintain the regularity of random variables or processes when we employ the
 localization techniques. Under some additional assumptions on the growth
and regularity for the process $Y$, we prove that ${\tau}_{Q}\wedge t$ is
quasi-continuous if $Q$ satisfies the exterior ball condition (see Section 3
for the definition). Furthermore, we show that ${\tau}_{Q}$ itself is
quasi-continuous if $Q$ is also bounded.

Our approach consists two key ingredients. One is to prove that ${\tau}%
_{Q}={\tau}_{\overline{Q}}$ q.s.$\ $(we say that a property holds
\textquotedblleft quasi-surely\textquotedblright\ (q.s.) if it holds $P$-a.s.
for each $P\in\mathcal{P}$) in Proposition \ref{Myth3.7}, where
\[
{\tau}_{\overline{Q}}(\omega):=\inf\{t\geq0:Y_{t}(\omega)\in\overline{Q}^{c}\}, \ \ \ \text{for}%
\ \omega\in\Omega.
\]
This is done by extending the auxiliary function argument in Lions and
Menaldi \cite{LM} to the case that the quadratic variation of $Y$ has possibly
unbounded rate of change and utilizing the tool of regular conditional
probability distributions of Stroock and Varadhan \cite{SV}. The other key
ingredient is to investigate the semi-continuities of ${\tau}_{Q}$ and ${\tau
}_{\overline{Q}}$ when the process $Y$ is continuous in $(\omega,t)$ and apply
a downward monotone convergence theorem for sets, see Lemma \ref{Myle3.1} and
Proposition \ref{Myth3.18}. First from the semi-continuities of exit times,
take in to account the regularity assumption on $Y$, we deduce that ${\tau
}_{Q}\wedge t$ is q.s. continuous on nearly all the domain $\Omega$. Then the
semi-continuities of ${\tau}_{Q}$ and ${\tau}_{\overline{Q}}$ allow us to use
the downward convergence theorem for upper capacity (see Song \cite{Song})
 to obtain an open set, on the complement of which ${\tau}_{Q}\wedge t$ is
continuous in $\omega$.

The rest of the paper is devoted to the study of the regularity for processes
needed for the above quasi-continuity of exit times. We give a
characterization theorem on the regularity of processes, thus generalized the
one for random variables in \cite{D-H-P}. We also investigate the
quasi-continuity of stopped processes when the stopping rule is a
quasi-continuous stopping time. Via the characterization theorem, we obtain
some typical examples of multi-dimensional nonlinear semimartingales
satisfying our assumptions, including $G$-martingales,
solutions of stochastic differential equations driven by $G$-Brownian motion and the canonical processes under a family of so-called semimartingale measures. We present at the end of the paper several
counterexamples to show that the exit times may not be quasi-continuous when our
assumptions are violated.

The paper is organized as follows. In Section 2, we recall the probabilistic
framework of nonlinear expectation and nonlinear semimartingales. The main
results on quasi-continuity of exit times for nonlinear semimartingales are
stated in Section 3. Section 4 is devoted to the research on the regularity of
 processes. Finally, in Section 5, we give some examples and counterexamples.

\section{Nonlinear expectation on the path space}

We present the basic setting and notations of the nonlinear expectation and nonlinear semimartingales. More relevant results can be found in \cite{Peng2007,peng2008,Peng 1,STZ}.

Let $\Omega:=C([0,\infty);\mathbb{R}^{k})$ be the space of all $\mathbb{R}%
^{k}$-valued continuous paths $(\omega_{t})_{t\geq0}$, equipped with the
distance
\[
\rho(\omega^{1},\omega^{2}):=\sum_{N=1}^{\infty}2^{-N}[(\sup_{t\in\lbrack
0,N]}|\omega_{t}^{1}-\omega_{t}^{2}|)\wedge1],
\]
Let $B_{t}(\omega):=\omega_{t}$ for $\omega\in\Omega$, $t\geq0$ be the
canonical process and $\mathcal{F}_{t}:=\sigma\{B_{s}:s\leq t\}$ for $t\geq0$
be the natural filtration of $B$. We denote $\mathcal{F}:=(\mathcal{F}%
_{t}\mathcal{)}_{t\geq0}.$ A mapping $\tau:\Omega\rightarrow[0,\infty]$ is called a stopping time if $\{\tau\leq t\}\in
\mathcal{F}_{t}$ for each $t\geq0$.
Let $\mathcal{P}$ be a family of probability measures on $(\Omega
,\mathcal{B}(\Omega))$, where $\mathcal{B}(\Omega)$ is the $\sigma$-algebra of all Borel sets. We set
\[
\mathcal{L}(\Omega):=\{X\in\mathcal{B}(\Omega):E_{P}[X]\ \text{exists for each
}\ P\in\mathcal{P}\}.
\]
We define the corresponding upper expectation by
\begin{equation}
\mathbb{\hat{E}}[X]:=\sup_{P\in\mathcal{P}}E_{P}[X]\in [-\infty,\infty], \ \ \ \text{for}%
\ X\in\mathcal{L}(\Omega).
\end{equation}
Then it is easy to check that the triple $(\Omega,\mathcal{L}(\Omega
),\mathbb{\hat{E}})$ forms a sublinear expectation space (see \cite{Peng 1}
for the definition).

For this $\mathcal{P}$, we define the corresponding upper capacity%
\[
c(A):=\sup_{P\in\mathcal{P}}P(A),\ \   \ \text{for}\ A\in\mathcal{B}(\Omega).
\]
A set $A\in\mathcal{B}(\Omega)$ is said to be polar if $c(A)=0$. We say a
property holds q.s. (quasi-surely) if it holds outside a polar set. In the
following, we do not distinguish two random variables if they coincide q.s.

We define the $L^{p}$-norm of random variables as $||X||_{p}:=(\mathbb{\hat
{E}}[|X|^{p}])^{\frac{1}{p}}$ for $p\geq1$ and set
\[
{L}^{p}(\Omega):=\{X\in\mathcal{B}(\Omega):||X||_{{p}}<\infty\}.
\]
Then ${L}^{p}(\Omega)$ is a Banach space under the norm $||\cdot||_{{p}}$. Let
$C_{b}(\Omega)$ be the space of all bounded, continuous functions on $\Omega$.
We denote the corresponding completion under the norm $||\cdot||_{p}$ by ${L}%
_{C}^{p}(\Omega)$.

\begin{definition}
A real function $X$ on $\Omega$ is said to be quasi-continuous if for each
$\varepsilon>0$, there exists an open set $O\subset\Omega$ with
$c(O)<\varepsilon$ such that $X|_{O^{c}}$ is continuous.
\end{definition}

\begin{definition}
We say that $X:\Omega\mapsto\mathbb{R}$ has a quasi-continuous version if
there exists a quasi-continuous function $Y:\Omega\mapsto\mathbb{R}$ such that
$X=Y$ q.s.
\end{definition}

The following result characterizes the space ${L}_{C}^{p}(\Omega)$ in the
measurable and integrable sense, which can be seen as a counterpart of Lusin's
theorem in the nonlinear expectation theory.

\begin{theorem}
[\cite{D-H-P}]\label{LG-ch}For each $p\geq1$, we have
\[
L_{C}^{p}(\Omega)=\{X\in\mathcal{B}(\Omega)\ :\ \ \lim\limits_{N\rightarrow
\infty}\mathbb{\hat{E}}[|X|^{p}I_{\{|X|\geq N\}}]=0\ \text{and}\ X\ \text{has
a quasi-continuous version}\}.
\]

\end{theorem}

Moreover, we have the following monotone convergence results. It worthy noting
that different from the linear case, the downward convergence is quite restrictive.

\begin{proposition}
[\cite{D-H-P,Song}]\label{Myth2.1} Suppose $X_{n}$, $n\geq1$ and $X$ are
$\mathcal{B}(\Omega)$-measurable.

\begin{description}
\item[\rm{(1)}] Assume $X_{n}\uparrow X$ q.s. on $\Omega$ and $E_{P}[X_{1}%
^{-}]<\infty$ for all $P\in\mathcal{P}$. Then $\mathbb{\hat{E}}[X_{n}%
]\uparrow\mathbb{\hat{E}}[X].$

\item[\rm{(2)}] Assume $\mathcal{P}$ is weakly compact.

\begin{itemize}
\item[\rm{(a)}] If $\{X_{n}\}_{n=1}^{\infty}$ in ${L}_{C}^{1}(\Omega)$ satisfies
that $X_{n}\downarrow X$ q.s., then $\mathbb{\hat{E}}[X_{n}]\downarrow
\mathbb{\hat{E}}[X].$

\item[\rm{(b)}] For each closed set $F\in\mathcal{B}(\Omega)$, $c(F)=\inf
\{c(O):\ O\makebox{ open in}\ \mathcal{B}(\Omega),\ F\subset O\}.$
\end{itemize}
\end{description}
\end{proposition}

\begin{definition}
An $\mathcal{F}$-adapted process $Y=(Y_{t})_{t\geq0}\ $is called a
$\mathcal{P}$-martingale $($$\mathcal{P}$-supermartingale, $\mathcal{P}%
$-submartingale, $\mathcal{P}$-semimartingale resp.$)$ if it is a martingale
$($supermartingale, submartingale, semimartingale resp.$)$ under each
$P\in\mathcal{P}$.
\end{definition}

The following is the quasi-continuity concept for processes, which is slightly
different from the one for random variables.

\begin{definition}
[\cite{Song1,Song}]We say that a process $F=(F_{t})_{t\geq0}$ is
quasi-continuous $($on $\Omega\times[0,\infty)$$)$ if for each $\varepsilon
>0$, there exists an open set $G\subset\Omega$ with $c(G)<\varepsilon$ such
that $F_{\cdot}(\cdot)$ is continuous on $G^{c}\times\lbrack0,\infty)$.
\end{definition}

\begin{remark}
\label{Re2.1}
\upshape{ From the definition, it is easy to see that, if the process
$F=(F_{t} )_{t\geq0}$ is quasi-continuous $($in the process setting$)$, then for each $t$, the random
variable $F_{t}$ is quasi-continuous $($in the random variable setting$)$.}
\end{remark}
\begin{definition}
Let $X,Y:\Omega\times[0,\infty)\rightarrow\mathbb{R}$ be two processes. We say $X$ is a modification of $Y$ if  for each $t\in [0,\infty)$, $X_t=Y_t$ q.s. We say $X$ is indistinguishable from $($or a version of$)$ $Y$ if $X_t=Y_t$  for all  $t\in [0,\infty)$, q.s.
\end{definition}
\section{Exit times for multi-dimensional nonlinear semimartingales}

Let $Y$ be a $d$-dimensional continuous $\mathcal{P}$-semimartingale under a
given weakly compact family $\mathcal{P}$ of probability measures. Assume
that, under each $P\in\mathcal{P}$, we have the decomposition $Y_{t}=M_{t}%
^{P}+A_{t}^{P},$ where $M_{t}^{P}$ is a $d$-dimensional continuous local
martingale and $A_{t}^{P}$ is a $d$-dimensional finite-variation process. We
also denote by $\langle Y\rangle^{P}=\langle M^{P}\rangle^{P}$ the quadratic
variation under $P,\ $and shall often omit the superscript $P$ on
$\langle\hspace{1mm}%
\cdot\hspace{1mm}\rangle$ when there is no danger of ambiguity.

\subsection{Quasi-continuity of exit times}

For each set $D\subset\mathbb{R}^{d}$, we define the exit times of $Y$ from
$D$ by
\[
{\tau}_{D}(\omega):=\inf\{t\geq0:Y_{t}(\omega)\in D^{c}\},\ \ \ \text{for}%
\ \omega\in\Omega.
\]

\begin{definition}
Let $E$ be a metric space. We say that a function $f:E\rightarrow
[-\infty,\infty] $ is upper $($lower resp.$)$
semi-continuous if for each $x_{0}\in E$,
\[
\limsup_{x\rightarrow x_{0}}f(x)\leq f(x_{0})\quad(\liminf_{x\rightarrow
x_{0}}f(x)\geq f(x_{0}) \ \text{resp.}).
\]

\end{definition}

\begin{definition}
An open set $O$ is said to satisfy the exterior ball condition at
$x\in\partial O$ if there exists an open ball $U(z,r)$ with center $z$ and
radius $r$ such that $U(z,r)\subset O^{c}$ and $x\in\partial U(z,r)$. If every boundary point
$x\in\partial O$ satisfies the exterior ball condition, then $O$ is said to satisfy the exterior ball condition.
\end{definition}

Given an open set $Q$ in $\mathbb{R}^{d}$, we denote
\begin{equation}
\label{eq3.1}\Omega^{\omega}=\{\omega^{\prime} \in\Omega:\omega_{t}^{\prime
}=\omega_{t} \ \text{on}\ [0,\tau_{Q}(\omega)]\},\ \ \ \text{for each}%
\ \omega\in\Omega.
\end{equation}

In this section, we shall mainly deal with nonlinear semimartinales $Y$
possessing a local growth condition at the boundary:

\begin{description}
\item[$(H)$] Given each $P\in\mathcal{P}$. For $P$-a.s. $\omega$ such that
${\tau}_{Q}(\omega)<\infty,$ there exist some stopping time $\sigma^{\omega}$
and constants $\lambda^{\omega},\varepsilon^{\omega}>0$ (these three
quantities may depend on $P,\omega$ but are supposed to be uniform for all $\omega^{\prime}\in
\Omega^{\omega}$) so that for $\omega^{\prime}\in\Omega^{\omega},$

\begin{description}
\item[\rm{{{(i)}}}] $\sigma^{\omega}(\omega^{\prime})>0$;

\item[\rm{{{(ii)}}}] On the interval of $t\in\lbrack0,\sigma^{\omega
}(\omega^{\prime})\wedge(\tau_{{\overline{Q}}}(\omega^{\prime})-\tau_{{{Q}}%
}(\omega^{\prime}))]$, it holds that

\begin{itemize}
\item[(a)] {Non-degeneracy:} $d\langle M^{P}\rangle_{{\tau}_{Q}(\omega
)+t}(\omega^{\prime})\geq\lambda^{\omega}\text{tr}[d\langle M^{P}%
\rangle_{_{{\tau}_{Q}(\omega)+t}}(\omega^{\prime})]I_{d\times d}%
$\ \ \text{and} \ \ $\text{tr}[d\langle M^{P}\rangle_{_{{\tau}_{Q}(\omega)+t}}%
(\omega^{\prime})]>0,$

\item[(b)] {Controllability:} $\text{tr}[d\langle M^{P}\rangle_{_{{\tau}%
_{Q}(\omega)+t}}(\omega^{\prime})]\geq\varepsilon^{\omega}|dA_{_{{\tau}%
_{Q}(\omega)+t}}^{P}(\omega^{\prime})|.$
\end{itemize}
\end{description}
\end{description}

The following theorem is the main result of this section concerning the regularity of exit times.

\begin{theorem}
\label{Myth3.5} Let $Q$ be an open set satisfying the exterior ball condition
and let $\Omega^{\omega}$ be defined as in (\ref{eq3.1}). Suppose that $Y$ is
quasi-continuous and satisfies the local growth condition $(H)$. Then for any
$\delta>0$, there exists an open set $O\subset\Omega$ such that $c(O)\leq
\delta$ and on $O^{c}$, it holds that:

\begin{description}
\item[\rm{{(i)}}] ${\tau}_{Q}$ is lower semi-continuous and ${\tau
}_{\overline{Q}}$ is upper semi-continuous;

\item[\rm{{(ii)}}] ${\tau}_{Q}={\tau}_{\overline{Q}}$.
\end{description}
\end{theorem}

\begin{remark}
\upshape{
	Let us explain the meaning of the three inequalities in {\rm (ii)} of the condition {$ (H)$}.
	
	\begin{description}
		\item[{\rm (a)}] We first give a general discussion. For two $($signed$)$ measures
	$\mu_{1}$ and  $\mu_{2}$ on some sub-interval $I$ of $\mathbb{R}$, by $d\mu_{1}\geq
	d\mu_{2}$ we mean  that $\mu_{1}(A)\geq\mu_{2}(A)$ for each $A\in
	\mathcal{B}(I).$ If  $\mu_{i},i=1,2,$ are Lebesgue-Stieltjes
	measures corresponding to  finite-variation functions $f_{i}$  respectively,
	$d\mu_{1}\geq d\mu_{2}$ is  equivalent to the assertion that $f_{1}-f_{2}$ is non-decreasing.
	
Such a discussion obviously also holds for the more general case that the two measures $\mu
_{1}$  and $\mu_{2}$ take $\mathbb{S}(d)$-values, where
$\mathbb{S}(d)$ is the set of $d\times d$ symmetric matrices endowed with the
usual  order, i.e., for $\gamma_1,\gamma_2\in \mathbb{S}(d)$ we write $\gamma_1\geq \gamma_2$ if $\gamma_1-\gamma_2$ is nonnegative definite.
\item[{\rm (b)}] On the interval of $t\in[0,\sigma^{\omega}(\omega^{\prime})\wedge (\tau_{{\overline{Q}}}(\omega')-\tau_{{{Q}}}(\omega'))]$, the Lebesgue-Stieltjes measures $d\langle M^{P}\rangle_{{\tau}_{Q}(\omega)+t}(\omega^{\prime})$, $\text{tr}[d\langle M^{P}\rangle_{_{{\tau}_{Q}(\omega)+t}}(\omega^{\prime})]$ and $dA_{_{{\tau}%
		_{Q}(\omega)+t}}^{P}(\omega^{\prime})$  are generated by increasing functions $t\rightarrow \langle M^{P}\rangle_{{\tau}_{Q}(\omega)+t}(\omega^{\prime})$, $t\rightarrow\text{tr}[\langle M^{P}\rangle_{_{{\tau}_{Q}(\omega)+t}}(\omega^{\prime})]$ and $t\rightarrow A_{_{{\tau}%
		_{Q}(\omega)+t}}^{P}(\omega^{\prime})$ respectively.
	
The first inequality in (a) of  $(H)$ means that the
measure $d\langle M^{P}\rangle_{{\tau}_{Q}(\omega)+t}(\omega^{\prime})$  is no smaller than $\lambda^{\omega}\text{tr}[d\langle M^{P}\rangle_{_{{\tau}_{Q}(\omega)+t}}(\omega^{\prime})]I_{d\times d}$, and the one in (b) of  $(H)$ has a similar meaning that the measure $\text{tr}[d\langle M^{P}\rangle_{_{{\tau}_{Q}(\omega)+t}}(\omega^{\prime})]$ is no smaller than $\varepsilon^{\omega}|dA_{_{{\tau}_{Q}(\omega)+t}}^{P}(\omega^{\prime})|,$ where $|dA_{_{{\tau}_{Q}(\omega)+t}}^{P}(\omega^{\prime})|$ is the total variation measure of $dA_{_{{\tau}_{Q}(\omega)+t}}^{P}(\omega^{\prime})$.
{Moreover, the second inequality in (a) of $(H)$ means  that the function $t\rightarrow \text{tr}[\langle M^{P}\rangle_{_{{\tau}_{Q}(\omega)+t}}(\omega^{\prime})]$ that generates the measure $\text{tr}[d\langle M^{P}\rangle_{_{{\tau}_{Q}(\omega)+t}}(\omega^{\prime})]$ is strictly increasing.}
\end{description}
}
\end{remark}

\begin{remark}\label{Myrem3.1}
\upshape{\begin{description}
			\item[{\rm (i)}] A simple and sufficient condition of {$ (H)$} is the case that $\lambda,\varepsilon$ are independent of $\omega$ and the growth condition is global, i.e.,
	\begin{description}
	\item[{$ (H')$}] for each $P$, there exist constants $\lambda,\varepsilon>0$ $($may depend on $P$$)$
		such that $d\langle M^{P}\rangle_{t}\geq\lambda$tr$[d\langle M^{P}\rangle
		_{t}]I_{d\times d}$, tr$[d\langle M^{P}\rangle_{t}]>0$ and tr$[d\langle M^{P}\rangle_{t} ]\geq\varepsilon
		|dA_{t}^{P}|$  on $[0,\tau_{{\overline{Q}}}]$, $ P$-a.s.
	\end{description}
	Indeed, we can take $\sigma^\omega\equiv t$ for any given $t>0$ in this situation, and then {$ (H)$} holds.

\item[{\rm (ii)}]	If for each $P$, it holds that $\lambda I_{d\times d}\leq \frac{d\langle M^{P}\rangle_{t}}{dt}\leq\Lambda I_{d\times
		d}$ and $|\frac{dA_{t}^{P}}{dt}|\leq C$  on $[0,\tau_{{\overline{Q}}}]$,  for some
	constants $0<\lambda\leq\Lambda,C\geq0$ $($may depend on $P$$)$, $P$-a.s., then {$ (H')$} is satisfied.
\end{description}
}
\end{remark}

\begin{remark}
\upshape{
	We discuss two special situations mainly based on the condition {$(H')$}. Similar results hold for {$ (H)$} by a straightforward modification. Since the symbols for the latter is more complicated and so is omitted.
}

\begin{description}
\item[\rm{(i)}] If $Y$ is a $\mathcal{P}$-martingale, i.e., $A^{P}%
\equiv0$ for each $P\in\mathcal{P}$, then the inequality tr$[d\langle M^{P}\rangle_{t} ]\geq\varepsilon
|dA_{t}^{P}|$ in {$(H^{\prime})$} holds trivially by taking  $\varepsilon=1$ .

\item[\rm{(ii)}] When $d=1,$ the inequality $d\langle M^{P}\rangle_{t}%
\geq\lambda$tr$[d\langle M^{P}\rangle_{t}]I_{d\times d}$ in {$(H^{\prime})$}
is just $d\langle M^{P}\rangle_{t}\geq\lambda d\langle M^{P}\rangle_{t},$ and
thus  holds for $\lambda= 1.$ If moreover $A^{P}\equiv0$, then
{$(H^{\prime})$} reduces to $d\langle M^{P}\rangle_{t}>0$ on $[0,\tau
_{{\overline{Q}}}]$.
\end{description}
\end{remark}

Before presenting the proof, we shall state a direct consequence of Theorem
\ref{Myth3.5} concerning the quasi-continuity of exit times. Note that ${\tau
}_{Q}$ and ${\tau}_{\overline{Q}}$ may take the value $+\infty$. The fact that
${\tau}_{Q}$ is lower semi-continuous, ${\tau}_{\overline{Q}}$ is upper
semi-continuous and ${\tau}_{Q}={\tau}_{\overline{Q}}$ does not imply that
${\tau}_{Q}$ and ${\tau}_{\overline{Q}}$ are continuous. In general, we can
get the quasi-continuity by a truncation manipulation as follows.

\begin{corollary}
\label{Myth3.4} Assume that the conclusion of Theorem \ref{Myth3.5} is true
for ${\tau}_{Q}$ and ${\tau}_{\overline{Q}}$.

\begin{description}
\item[\rm{(i)}] If $X$ is a quasi-continuous random variable, then ${\tau
}_{Q}\wedge X$ and ${\tau}_{\overline{Q}}\wedge X$ are both quasi-continuous.

\item[\rm{(ii)}] If $X\in L_{C} ^{1}(\Omega)$, then ${\tau}_{Q}\wedge X$
and ${\tau}_{\overline{Q}}\wedge X$ both belong to $L_{C}^{1}(\Omega)$.
\end{description}
\end{corollary}

\begin{proof}
(i) By assumption, we can find an open set $O_{1}\subset\Omega$ such that
$c(O_{1})\leq\varepsilon$ and $X$ is continuous on $(O_{1})^{c}$. Moreover,
from Theorem \ref{Myth3.5}, we can choose an open set $O_{2}\subset\Omega$
such that $c(O_{2})\leq\varepsilon$ and on $(O_{2})^{c}$, ${\tau}_{Q}$ and
${\tau}_{\overline{Q}}$ are lower  and upper semi-continuous
respectively, and ${\tau}_{{Q}}={\tau}_{\overline{Q}}$. Denote $O=O_{1}\cup
O_{2}$. Then $c(O)\leq2\varepsilon$ and on $O^{c}$, it holds that ${\tau}%
_{{Q}}\wedge X:\Omega\rightarrow\mathbb{R}$ is lower semi-continuous, ${\tau
}_{\overline{Q}}\wedge X:\Omega\rightarrow\mathbb{R}$ is upper
semi-continuous, and
\[
{\tau}_{{Q}}\wedge X={\tau}_{\overline{Q}}\wedge X.
\]
From this, we deduce that ${\tau}_{{Q}}\wedge X$ and ${\tau}_{\overline{Q}%
}\wedge X$ are continuous on $O^{c}$.

(ii) From (i), ${\tau}_{Q}\wedge X$ is quasi-continuous. Noting that $|{\tau
}_{Q}\wedge X|\leq|X|$, then
\[
\hat{\mathbb{E}}[|{\tau}_{Q}\wedge X|I_{\{|{\tau}_{Q}\wedge X|>k\}}]\leq
\hat{\mathbb{E}}[|X|I_{\{|X|>k\}}]\rightarrow0,\ \ \ \text{as}\ k\rightarrow
\infty.
\]
Now the desired result follows from Theorem \ref{LG-ch}.
\end{proof}

\begin{remark}
\upshape{
Typically, we shall often take $X\equiv t$, for any fixed $t$, in the above corollary.

 Assume that $d=1$ and
$Y$ is a one-dimensional $\mathcal{P}$-martingale. Then from Corollary
\ref{Myth3.4}, we deduce that ${\tau}_{Q}\wedge t$ is quasi-continuous if
$d\langle M^{P}\rangle_{t}>0$ $P$-a.s., for each $P\in\mathcal{P}$, and $Q$
satisfies the exterior ball condition. In particular, if we take
$Q=(-\infty,a)$ for $a\in\mathbb{R}$, then
\[
{\tau}_{{Q}}(\omega)=\inf\{t\geq0:Y_{t}(\omega)>a\}
\]
and we get the result in \cite{Song}.
}
\end{remark}

Now we proceed to the proof of Theorem \ref{Myth3.5}. We first present a
result which shows that $Y$ originating at the boundary point of $Q$ with
exterior ball will exit $\overline{Q}$ immediately.

\begin{proposition}
\label{Myth3.1} Let $Q$ be an open set satisfying the exterior ball condition
at some $x\in\partial Q$. Assume $P$ is a probability measure such that
$Y=M^{P}+A^{P}$ is a continuous semimartingale satisfying $Y_{0}=x$ P-$a.s$
and the following local growth assumption at $x$:

\begin{description}
\item[$(A)$] There exists some stopping time $\sigma>0$ $P$-a.s. and constants
$\lambda,\varepsilon>0$ such that $d\langle M^{P}\rangle_{t}\geq
\lambda\text{tr}[d\langle M^{P}\rangle_{t}]I_{d\times d},$ $\text{tr}[d\langle
M^{P}\rangle_{t}]>0$ and $\text{tr}[d\langle M^{P}\rangle_{t}]\geq
\varepsilon|dA_{t}^{P}|$ on $[0,\sigma\wedge\tau_{{\overline{Q}}}]$,
$P\text{-a.s.}$
\end{description}
Then we have $\tau_{\overline{Q}}=0$ $P$-a.s., i.e., $P$-a.s. for each
$\delta>0$, there exists a point $t\in(0,\delta]$ such that $Y_{t}\in
\overline{Q}^{c}$.
\end{proposition}

\begin{proof}
Let $U(z,r)$ be the exterior ball of $Q$ at $x$. We set $h(y):=e^{-k|y-z|^{2}
}$, where the constant $k$ will be determined in the sequel. Then
\begin{align*}
&  D_{y}h(y)=-2k(y-z)e^{-k|y-z|^{2}},\\
&  D_{yy}^{2}h(y)=(4k^{2}(y_{i}-z_{i})(y_{j}-z_{j})-2k\delta_{ij}
)e^{-k|y-z|^{2}}=(4k^{2}(y-z)(y-z)^{T}-2kI_{d\times d})e^{-k|y-z|^{2}}.
\end{align*}

Let $\langle\cdot,\cdot\rangle$ be the Euclidian scalar product for vectors
and matrices. Let any $R\geq 2r$ be given. By the assumption, for $P$-a.s. $\omega
$, on $[0,\sigma]$, we have for all $y\in U(x,R)\cap\overline{Q}$,
\begin{equation}%
\begin{split}
&  \langle D_{yy}^{2}h(y),d\langle M^{P}\rangle_{t}\rangle+2\langle
D_{y}h(y),dA_{t}^{P}\rangle\\
&  =(\langle4k^{2}(y-z)(y-z)^{T},d\langle M^{P}\rangle_{t}\rangle
-\langle2kI_{d\times d},d\langle M^{P}\rangle_{t}\rangle-4k\langle
(y-z),dA_{t}^{P}\rangle)e^{-k|y-z|^{2}}\\
&  \geq(\langle4k^{2}(y-z)(y-z)^{T},\lambda\text{tr}[d\langle M^{P}\rangle
_{t}]I_{d\times d}\rangle-\langle2kI_{d\times d},d\langle M^{P}\rangle
_{t}\rangle-4k\langle(y-z),dA_{t}^{P}\rangle)e^{-k|y-z|^{2}}\\
&  \geq(4\lambda k^{2}|y-z|^{2}\text{tr}[d\langle M^{P}\rangle_{t}
]-4k|y-z||dA_{t}^{P}|-2k\text{tr}[d\langle M^{P}\rangle_{t}])e^{-k|y-z|^{2}}\\
&  \geq((4\lambda k^{2}r^{2}-2k)\text{tr}[d\langle M^{P}\rangle_{t}
]-4k(R+r)\frac{1}{\varepsilon}\text{tr}[d\langle M^{P}\rangle_{t}
])e^{-k|y-z|^{2}}\\
&  =(((4\lambda k^{2}r^{2}-2k)-4k(R+r)\frac{1}{\varepsilon})\text{tr}[d\langle
M^{P}\rangle_{t}])e^{-k|y-z|^{2}}.
\end{split}
\label{Myeq3.11}%
\end{equation}
Here we have used the well-known matrix inequality that $\langle A_{1},B\rangle
\geq\langle A_{2},B\rangle$ if $A_{1},A_{2},B\in\mathbb{S}(d)$ such that
$A_{1}\geq A_{2}$ and $B\geq0$ (recall that $\mathbb{S}(d)$ is the set of
$d\times d$ symmetric matrices with the usual order).

Since $M^{P}$ is a local martingale, we can find a stopping times $\sigma
_{1}>0$ such that $M_{\cdot\wedge\sigma_{1}}^{P}$ is a square-integrable
martingale. For symbol simplicity, we still denote $\sigma\wedge\sigma_{1}$ by
$\sigma.$ For any given $t>0$, applying It\^{o}'s formula, we obtain
\begin{align*}
h(Y_{\tau_{\overline{Q}}\wedge{\tau}_{U(x,R)}\wedge\sigma\wedge t})-h(x)  &
=\int_{0}^{\tau_{\overline{Q}}\wedge{\tau}_{U(x,R)}\wedge\sigma\wedge
t}\langle D_{y}h(Y_{s}),dM_{s}^{P}\rangle+\int_{0}^{\tau_{\overline{Q}}
\wedge{\tau}_{U(x,R)}\wedge\sigma\wedge t}\langle D_{y}h(Y_{s}),dA_{s}
^{P}\rangle\\
&  \ \ \ +\frac{1}{2}\int_{0}^{\tau_{\overline{Q}}\wedge{\tau}_{U(x,R)}
\wedge\sigma\wedge t}\langle D_{yy}^{2}h(Y_{s}),d\langle M^{P}\rangle
_{s}\rangle.
\end{align*}
Taking expectation on both sides, we get
\[
{{E}}_{P}[\int_{0}^{\tau_{\overline{Q}}\wedge{\tau}_{U(x,R)}\wedge\sigma\wedge
t}(\frac{1}{2}\langle D_{yy}^{2}h(Y_{s}),d\langle M^{P}\rangle_{s}
\rangle+\langle D_{y}h(Y_{s}),dA_{s}^{P}\rangle)]={{E}}_{P}[h(Y_{\tau
_{\overline{Q}}\wedge{\tau}_{U(x,R)}\wedge\sigma\wedge t})-h(x)]\leq0,
\]
since $h(y)-h(x)\leq0$ for each $y\in(U(z,r))^{c}$. Combining this with
inequality (\ref{Myeq3.11}), we get
\[
{{E}}_{P}[\int_{0}^{\tau_{\overline{Q}}\wedge{\tau}_{U(x,R)}\wedge\sigma\wedge
t}(((2\lambda k^{2}r^{2}-k)-2k(R+r)\frac{1}{\varepsilon})\text{tr}[d\langle
M^{P}\rangle_{s}])e^{-k|Y_{s}-z|^{2}}]\leq0.
\]
This can be rewritten as
\[
{{E}}_{P}[\int_{0}^{\tau_{\overline{Q}}\wedge{\tau}_{U(x,R)}\wedge\sigma\wedge
t}(((2\lambda kr^{2}-1)\varepsilon-2(R+r))\text{tr}[d\langle M^{P}\rangle
_{s}])e^{k((R+r)^{2}-|Y_{s}-z|^{2})}]\leq0.
\]
If $P(\tau_{\overline{Q}}>0)>0$, then $P(\tau_{\overline{Q}} \wedge{\tau
}_{U(x,R)}\wedge\sigma\wedge t>0)>0$. In view of
\[
((2\lambda kr^{2}-1)\varepsilon-2(R+r))e^{k((R+r)^{2}-|Y_{s}-z|^{2})}
\uparrow\infty,\ \ \ \text{as}\ k_{0}\leq k\rightarrow\infty,\text{ for some
}k_{0}>0,
\]
we can apply the classical monotone convergence theorem to obtain
\[
\lim_{k\rightarrow\infty}{{E}}_{P}[\int_{0}^{\tau_{\overline{Q}}\wedge{\tau
}_{U(x,R)}\wedge\sigma\wedge t}(((2\lambda kr^{2}-1)\varepsilon
-2(R+r))\text{tr}[d\langle M^{P}\rangle_{s}])e^{k((R+r)^{2}-|Y_{s}-z|^{2}
)}]=\infty,
\]
which is a contradiction. So we must have ${\tau_{\overline{Q}}}=0$. The proof
is complete.
\end{proof}

\begin{remark}
\upshape{
			\begin{description}
				\item[{\rm (i)}]
		Surely the assumption $(A)$ is satisfied by the global growth condition that $(A)$ holds with $\sigma=\infty$.
		
		\item[{\rm (ii)}]	The presence of $\sigma$ should be understood by the observation that the phenomena of immediate exit from $\overline{Q}$ is a local behaviour which is determined by the path property of $Y$ near time $0$, i.e., the behaviour of $Y$ on $[0,\sigma]$.
	\end{description}}
\end{remark}

The immediate leaving property also holds for $Y$ with general initial points.

\begin{proposition}
\label{Myth3.7}Let $Y,Q$ be assumed as in Theorem \ref{Myth3.5}. Then
\begin{equation}
{\tau}_{Q}={\tau}_{\overline{Q}},\ \ \ \text{q.s.} \label{Myeq3.4}%
\end{equation}

\end{proposition}

\begin{proof}
Given any $P\in\mathcal{P}$. Observe that if $Y_{{0}}=x$ $P$-a.s. for some
$x\in\partial Q$, from Proposition \ref{Myth3.1}, we obviously have that
${\tau}_{\overline{Q}}={\tau}_{Q}=0,\ P$-a.s. If not, we will use the method
of regular conditional expectations to restart $Y$ at the boundary as following.

For $\mathcal{F}_{{\tau}_{Q}}$, from Theorem 1.3.4 in \cite{SV}, there exists
a regular conditional expectation $\{P^{\omega}\}_{{\omega\in}\Omega}$ such
that
\[
P^{\omega}(\Omega^{\omega})=1\ \ \text{and}\ \ E_{P}[\cdot|\mathcal{F}_{{\tau
}_{Q}}](\omega)=E_{P^{\omega}}[\cdot], \ \ \text{ for}\ P\text{-a.s.}\ \omega.
\]
If ${\tau}_{Q}(\omega)=\infty,$ it is obvious that ${\tau}_{\overline{Q}%
}(\omega)={\tau}_{Q}(\omega).$

For $P\text{-a.s.}\ \omega$, we have $\sigma^{\omega}>0$ $P^{\omega
}\text{-a.s}$. Moreover, for any given $\omega$, by Galmarino's test (see
\cite{RY}, Chap. I, Exercise 4.21 (3)), we have for $\omega^{\prime}\in\Omega^\omega$,
$Y_{t}(\omega)=Y_{t}(\omega^{\prime})$, $t\leq{\tau}_{Q}(\omega)$. This implies
that ${\tau}_{Q}(\omega)={\tau}_{Q}(\omega^{\prime})$. Thus, for $\omega$ such
that $P^{\omega}(\Omega^{\omega})=1$, under ${P^{\omega},}$ $\omega^{\prime
}\rightarrow{\tau}_{\overline{Q}}(\omega^{\prime})-{\tau}_{Q}(\omega^{\prime
})$ is also the exit time of $\omega^{\prime}\rightarrow(Y_{{\tau}_{Q}%
(\omega)+t}(\omega^{\prime}))_{t\geq0}$ from $\overline{Q}$.

Applying the following Lemma \ref{Myth3.8}, we deduce that for $P$-a.s.
$\omega$ such that ${\tau}_{Q}(\omega)<\infty,$ under ${P^{\omega},}$
$\omega^{\prime}\rightarrow(Y_{{\tau}_{Q}(\omega)+t}(\omega^{\prime}%
))_{t\geq0}$ is a semimartingale starting from $Y_{{\tau}_{Q}(\omega)}%
\in\partial Q$ and satisfying the assumption $(A)$ in Proposition
\ref{Myth3.1}. Therefore, by applying Proposition \ref{Myth3.1}, we obtain
\[
E_{P}[({\tau}_{\overline{Q}}-{\tau}_{Q})I_{\{{\tau}_{Q}<\infty\}}%
|\mathcal{F}_{{\tau}_{Q}}](\omega)=E_{{P^{\omega}}}[{\tau}_{\overline{Q}%
}-{\tau}_{Q}]I_{\{{\tau}_{Q}(\omega)<\infty\}}=0,\ \ \text{ for}\ P\text{-a.s.}%
\ \omega\text{. }%
\]

Summarizing the above, we get
\[
{\tau}_{\overline{Q}}={\tau}_{Q},\ \ \ P\text{-a.s.},
\]
which implies%
\[
{\tau}_{\overline{Q}}={\tau}_{Q},\ \ \ \text{q.s.}%
\]
This completes the proof.
\end{proof}

\begin{lemma}
\label{Myth3.8} Let $\tau:\Omega\rightarrow[0,\infty]$ be a stopping
time. Given a local martingale $(M_{t}^{P},\mathcal{F}_{t})_{t\geq0}$ under
some probability measure ${P}$. Let $\{{P^{\omega}\}}_{{\omega\in}\Omega}$ be
the corresponding regular conditional expectation of ${P}$ for $\mathcal{F}%
_{{\tau}}.$ Then for $P$-a.s.$\ \omega,$

\begin{description}
\item[\rm{(i)}] Under ${P^{\omega}}$, $\omega^{\prime}\rightarrow
(M_{t}^{P}(\omega^{\prime})-M_{{\tau}(\omega)\wedge t}^{P}(\omega^{\prime
}),\mathcal{F}_{t})_{t\geq0}$ is a local martingale, which can also be
restated as that $\omega^{\prime}\rightarrow(M_{{\tau}(\omega)+t}^{P}%
(\omega^{\prime}),\mathcal{F}_{{\tau}(\omega)+t})_{t\geq0}\ $is a local
martingale for ${\tau}(\omega)<\infty.$

\item[\rm{(ii)}] If ${\tau}(\omega)<\infty,$ then $\langle M_{\tau
(\omega)+\cdot}^{P}\rangle_{t}^{P^{\omega}}=\langle M^{P}\rangle_{\tau
(\omega)+t}^{P}$ for each $t\geq0,$ ${P^{\omega}}$-a.s. $($recall that when necessary, we use
the superscript $Q$ on the quadratic variation $\langle\hspace{1mm}%
\cdot\hspace{1mm} \rangle$ to indicate the dependence on a probability
$Q$$)$.
\end{description}
\end{lemma}

\begin{proof}
(i) \textit{Step 1.} If $M^{P}$ is a martingale under ${P,}$ then by Theorem
1.2.10 in \cite{SV}, for $P$-a.s.$\ \omega,$ $\omega^{\prime}\rightarrow
(M_{t}^{P}(\omega^{\prime})-M_{{\tau}(\omega)\wedge t}^{P}(\omega^{\prime
}))_{t\geq0}$ is a $\mathcal{F}_{t}$-martingale under ${P^{\omega}.}$

\textit{Step 2.} Now suppose that $M^{P}$ is a local martingale under ${P.}$
Let $T_{n}$ be localization sequence of stopping times for $M^{P}$ such that
$T_{n}\uparrow\infty$ ${P}$-a.s. and $(M_{t\wedge T_{n}}^{P})_{t\geq0}$ is a
martingale under $P$.  For any
given $n$, since $M_{t\wedge T_{n}}^{P}$ is a martingale under $P$, applying
Step 1 yields that $\omega^{\prime}\rightarrow M_{t\wedge T_{n}(\omega
^{\prime})}^{P}(\omega^{\prime})-M_{{\tau}(\omega)\wedge t\wedge T_{n}%
(\omega^{\prime})}^{P}(\omega^{\prime})$ is a martingale under ${P^{\omega},}$
for$\ P$-a.s.$\ \omega.$ Since
 $T_{n}\uparrow\infty$ ${P^{\omega}}$-a.s., for $P$-a.s.$\ \omega$,  we can find a set $N\subset\Omega$ such that
$P(N)=0$ and for $\omega\in N^{c},$ $\omega^{\prime}\rightarrow M_{t\wedge
T_{n}(\omega^{\prime})}^{P}(\omega^{\prime})-M_{{\tau}(\omega)\wedge t\wedge
T_{n}(\omega^{\prime})}^{P}(\omega^{\prime})$ is a martingale under
${P^{\omega}}$ for each $n$ and $T_{n}\uparrow\infty$ ${P^{\omega}}$-a.s. Let
any $\omega\in N^{c}$ be given such that ${\tau}(\omega)<\infty$ and
$\omega^{\prime}\rightarrow M_{({\tau}(\omega)+t)\wedge T_{n}(\omega^{\prime
})}^{P}(\omega^{\prime})$ is a $\mathcal{F}_{{\tau}(\omega)+t}$-martingale
under ${P^{\omega}}$ for each $n$. We define
\[
\sigma_{m}(\omega^{\prime}):=\inf\{t\geq0:|M_{{\tau}(\omega)+t}(\omega
^{\prime})|\geq m\},\ \ \ m\geq1,
\]
which is a $\mathcal{F}_{{\tau}(\omega)+t}$-stopping time. We claim that
$\sigma_{m}$ is a localization sequence for $M_{{\tau}(\omega)+t}^{P}.$
Indeed, note that
\[
\sup_{t\geq0}|M_{{\tau}(\omega)+t\wedge\sigma_{m}}^{P}|\leq m,
\]
then we can apply the dominated convergence theorem to derive that, for $s\leq
t,$%
\begin{align*}
E_{P^{\omega}}[M_{{\tau}(\omega)+t\wedge\sigma_{m}}^{P}|\mathcal{F}_{{\tau
}(\omega)+s}] &  =E_{P^{\omega}}[\lim_{n\rightarrow\infty}M_{({\tau}%
(\omega)+t\wedge\sigma_{m})\wedge T_{n}}^{P}|\mathcal{F}_{{\tau}(\omega)+s}]\\
&  =\lim_{n\rightarrow\infty}E_{P^{\omega}}[M_{({\tau}(\omega)+t\wedge
\sigma_{m})\wedge T_{n}}^{P}|\mathcal{F}_{{\tau}(\omega)+s}]\\
&  =\lim_{n\rightarrow\infty}M_{({\tau}(\omega)+s\wedge\sigma_{m})\wedge
T_{n}}^{P}\\
&  =M_{{\tau}(\omega)+s\wedge\sigma_{m}}^{P},
\end{align*}
where the third equality is due to the fact that $M_{({\tau}(\omega
)+t\wedge\sigma_{m})\wedge T_{n}}^{P}$ is a $\mathcal{F}_{{\tau}(\omega)+t}%
$-martingale by the classical optional sampling theorem. Therefore, $(M_{{\tau}%
(\omega)+t}^{P},\mathcal{F}_{{\tau}(\omega)+t})_{t\geq0}$ is a local
martingale under ${P^{\omega}}$.

(ii) Note that $(M_{t}^{P})^{2}-\langle M^{P}\rangle_{t}^{P}$ is a local
martingale under ${P}$. Then from Step 2 in (i), we obtain that for
$P\text{-a.s.}\ \omega,$ if $\tau(\omega)<\infty$, then $(M_{\tau(\omega
)+t}^{P})^{2}-\langle M^{P}\rangle_{\tau(\omega)+t}^{P}$ is also a local martingale
under ${P^{\omega}}$. This implies that
\[
\langle M_{\tau(\omega)+\cdot}^{P}\rangle_{t}^{P^{\omega}}=\langle
M^{P}\rangle_{\tau(\omega)+t}^{P}\ \ \ \text{for each}\ t\geq0,\ {P^{\omega}%
}\text{-a.s.},
\]
as desired.
\end{proof}

The following lemma concerns the semi-continuities of exit times when the
process is bi-continuous.

\begin{lemma}
\label{Myle3.1} Let $E$ be a metric space and $(\omega,t)\rightarrow
F_{t}(\omega)$ is a continuous mapping from $E\times\lbrack0,\infty
)\ $to$\ \mathbb{R}^{d}$. Define, for each set $D\subset\mathbb{R}^{d}$, the
exit times of $F$ from $D$ by
\[
{\sigma}_{D}(\omega):=\inf\{t\geq0:F_{t}(\omega)\in D^{c}\},\ \ \ \text{for}
\ \omega\in\Omega.
\]
Assume $Q$ is an open set. Then ${\sigma}_{Q}$ is lower semi-continuous and
${\sigma}_{\overline{Q}}$ is upper semi-continuous.
\end{lemma}

\begin{proof}
We first show that ${\sigma}_{\overline{Q}}$ is upper semi-continuous. For any
given $\omega\in E$, set $t_{0}:={\sigma}_{\overline{Q}}(\omega)$. Noting that
the case $t_{0}=\infty$ is trivial, we may assume that $t_{0}<\infty$. Then we
can find an arbitrarily small $\varepsilon>0$ such that $F_{t_{0}
+\varepsilon}(\omega)\in\overline{Q}^{c}$. Since $\overline{Q}^{c}$ is open,
there exists an open ball $U(F_{t_{0}+\varepsilon}(\omega),r)$ with center
$F_{t_{0}+\varepsilon}(\omega)$ and radius $r$ such that $U(F_{t_{0}
+\varepsilon}(\omega),r)\subset\overline{Q}^{c}$. For each ${\omega}^{\prime}$
whose distance with $\omega$ is sufficiently small, we will have
\[
F_{t_{0}+\varepsilon}({\omega}^{\prime})\in U(F_{t_{0}+\varepsilon}
(\omega),r)\subset\overline{Q}^{c}
\]
by the continuity of $F$. That is,
\[
{\sigma}_{\overline{Q}}({\omega}^{\prime})\leq t_{0}+\varepsilon,
\]
as desired.

Now we consider the second part. Given any $\omega\in E,$ we first prove the
assertion that for any given $t\in[0,\infty)$, if ${\sigma}_{Q}(\omega)\geq t,$
then
\begin{equation}
\liminf_{{\omega}^{\prime}\rightarrow{\omega}}{\sigma}_{Q}({\omega}^{\prime
})\geq{t.} \label{Myeq3.15}%
\end{equation}
If not, we can find a sequence $\omega^{n}\in E$ and $t_{n}\in\lbrack
0,t-\varepsilon]$ for some $\varepsilon>0$ such that
\[
\omega^{n}\rightarrow\omega\ \text{ and }\ F_{t_{n}}(\omega^{n})\in Q^{c}.
\]
We can extract a subsequence of $\{t_{n}\},$ which is still denoted by
$\{t_{n}\},$ such that $t_{n}\rightarrow t^{\prime}$ for some $t^{\prime}
\in\lbrack0,t-\varepsilon].$ Then by the continuity assumption on $F,$
\[
F_{t^{\prime}}(\omega)=\lim_{n\rightarrow\infty}F_{t_{n}}(\omega^{n})\in
Q^{c},
\]
which is a contradiction. Thus we have proved the assertion. Now set
$t_{0}:={\sigma}_{Q}(\omega).$ If $t_{0}<\infty,$ the conclusion follows from
taking ${t=t}_{0}$ in (\ref{Myeq3.15}). If $t_{0}=\infty,$ we can apply
(\ref{Myeq3.15}) to each $t<\infty$ to show that
\[
\liminf_{{\omega}^{\prime}\rightarrow{\omega}}{\sigma}_{Q}({\omega}^{\prime
})\geq{t},\ \ \text{ for every }{t>0,}
\]
which implies
\[
\liminf_{{\omega}^{\prime}\rightarrow{\omega}}{\sigma}_{Q}({\omega}^{\prime
})=\infty.
\]
The proof is now complete.
\end{proof}

Now we can complete the proof of Theorem \ref{Myth3.5}. For this
purpose, it suffice to prove the following proposition which is stated in a
slightly more general form, without the specific assumptions on $Q$ and $Y$ as in
Theorem \ref{Myth3.5}. It can be useful in the future work.

\begin{proposition}
\label{Myth3.18} Let $Q$ be an open set. Suppose that $Y$ is quasi-continuous
with exit times satisfying ${\tau}_{Q}={\tau}_{\overline{Q}}$ $q.s.$ Then the
conclusion in Theorem \ref{Myth3.5} holds: for any $\delta>0$, there exists
an open set $O\subset\Omega$ such that $c(O)\leq\delta$ and on $O^{c}$,
${\tau}_{Q}$ and ${\tau}_{\overline{Q}}$ are lower and upper semi-continuous respectively, and ${\tau}_{Q}={\tau}_{\overline{Q}}$.
\end{proposition}

\begin{proof}
Set $\Gamma=\{{\tau}_{Q}={\tau}_{\overline{Q}}\}$. Then $c(\Gamma^{c})=0$ by
the assumption. Since the process $Y$ is quasi-continuous, for any $\delta>0$,
we can find an open set $G\subset\Omega$ such that $c(G)\leq\frac{\delta}{2}$
and $Y$ is continuous on $G^{c}\times\lbrack0,\infty)$. From Lemma
\ref{Myle3.1}, ${\tau}_{Q}$ and ${\tau}_{\overline
	{Q}}$ are lower and upper semi-continuous respectively on $G^{c}$. Moreover, we can write the polar
set
\[
\Gamma^{c}\cap G^{c}=\{{\tau}_{Q}<{\tau}_{\overline{Q}}\} \cap G^{c}%
=\bigcup_{s<r;s,r\in\mathbb{Q}}(\{{\tau}_{Q}\leq s\} \cap\{{\tau}%
_{\overline{Q}}\geq r\})\cap G^{c}.
\]
For every $s,r,$ from the semi-continuities of ${\tau}_{Q}$ and ${\tau
}_{\overline{Q}}$ on $G^{c}$, we deduce that $(\{{\tau}_{Q}\leq s\}
\cap\{{\tau}_{\overline{Q}}\geq r\})\cap G^{c}$ is closed. Then according to
Proposition \ref{Myth2.1} (2) (b), there exists an open set with any given
small capacity such that
\[
O_{sr}\supset(\{{\tau}_{Q}\leq s\} \cap\{{\tau}_{\overline{Q}}\geq r\})\cap
G^{c}.
\]
From this, we can find an open set $O_{1}\supset\Gamma^{c}\cap G^{c}$ such
that $c(O_{1})\leq\frac{\delta}{2}$. Denote the open set $O=O_{1}\cup G.$ Then
on $O^{c}$, ${\tau}_{Q}$ is lower semi-continuous and ${\tau}_{\overline{Q}}$
is upper semi-continuous, and ${\tau}_{Q}={\tau}_{\overline{Q}}$.
\end{proof}


\begin{remark}
\upshape{
\label{Re3.1} In Proposition \ref{Myth3.1}, the condition in $(A)$ that there
exist some constant $\varepsilon>0$ such that
\begin{equation}
\text{tr}[d\langle M^{P}\rangle_{t} ]\geq \varepsilon|dA_{t}^{P}|,
\label{Myeq3.12}\end{equation}
can be relaxed in two one-dimensional cases.
Note that we use inequality $($\ref{Myeq3.12}$)$ to guarantee that, in (\ref{Myeq3.11}) in the proof
of Proposition \ref{Myth3.1},
\begin{equation}
\text{tr}[d\langle M^{P}\rangle_{t}]\geq\varepsilon\langle\frac{y-z}
{|y-z|},dA_{t}^{P}\rangle,\ \ \text{ for each }y\in\overline{Q}. \label{Myeq3.14}\end{equation}
Assume that $d=1$ and $Q=(-\infty,a)$ for some $a\in\mathbb{R}$. We take the
exterior ball $U(a+1,1)=(a,a+2)$. Then the inequality $($\ref{Myeq3.14}$)$ reduces
to
\begin{equation*}
d\langle M^{P}\rangle_{t}\geq\varepsilon\langle\frac{y-a-1}{|y-a-1|}
,dA_{t}^{P}\rangle,\ \ \text{ for each }y\leq a, \label{Myeq3.1}\end{equation*}
which is just
\begin{equation}
\label{Myeq3.17}d\langle M^{P}\rangle_{t}\geq-\varepsilon dA_{t}^{P}.
\end{equation}
Similar analysis shows that when $d=1$ and $Q=(a,+\infty)$ for some
$a\in\mathbb{R},$ the inequality $($\ref{Myeq3.14}$)$ reduces to
\begin{equation}
\label{Myeq3.16}d\langle M^{P}\rangle_{t}\geq\varepsilon dA_{t}^{P}.
\end{equation}
In these two situations respectively, we can  use $($\ref{Myeq3.17}) and
$($\ref{Myeq3.16}$)$ to replace (\ref{Myeq3.12}) and get the conclusion of Proposition \ref{Myth3.1}. We can also similarly modify the assumption $(H)$ in Theorem \ref{Myth3.5} and repeat the proofs as before, to recover all the corresponding results in this subsection.
}
\end{remark}

\subsection{Integrability of exit times}

When a certain integrability condition imposed, ${\tau}_{{Q}}$ and ${{\tau
}_{\overline{Q}}}$ themselves can be quasi-continuous.

\begin{theorem}
\label{Myth3} Assume that the conclusion of Theorem \ref{Myth3.5} is true for
${\tau}_{Q}$ and ${\tau}_{\overline{Q}}$, i.e., for any $\delta>0$, we can
find an open set $O\subset\Omega$ with $c(O)\leq\delta$ such that on $O^{c}$,
${\tau}_{Q}$ is lower semi-continuous, ${\tau}_{\overline{Q}}$ is upper
semi-continuous and ${\tau}_{Q}={\tau}_{\overline{Q}}$.

\begin{description}
\item[\rm{(i)}] If
\begin{equation}
c(\{{{\tau}_{\overline{Q}}}>k \})\rightarrow0, \ \ \ \text{as}\ k\rightarrow\infty,
\label{Myeq3.13}%
\end{equation}
then ${\tau}_{{Q}}$ and ${\tau}_{\overline{Q}}$ are quasi-continuous.

\item[\rm{(ii)}] Assume that
\begin{equation}
\hat{\mathbb{E}}[{{\tau}_{\overline{Q}}}I_{\{{{\tau}_{\overline{Q}}}%
>k\}}]\rightarrow0,\ \ \ \text{as}\ k\rightarrow\infty. \label{Myeq3.3}%
\end{equation}
Then ${\tau}_{{Q}}$ and ${\tau}_{\overline{Q}}$ both belong to $L_{C}%
^{1}(\Omega)$.
\end{description}
\end{theorem}

\begin{proof}
Since ${\tau}_{{Q}}={\tau}_{\overline{Q}}$ q.s., we may mainly prove the
conclusions for ${\tau}_{{Q}}$.

(i) By the assumption, we can choose an open set $O_{1}$ such that
$c(O_{1})\leq\varepsilon$ and on $(O_{1})^{c}$, ${\tau}_{{Q}}$ is lower
semi-continuous, ${\tau}_{\overline{Q}}$ is upper semi-continuous and$\ {\tau
}_{{Q}}={\tau}_{\overline{Q}}$. From (\ref{Myeq3.13}), we can also take $k$
sufficiently large such that $c(\{{\tau}_{{Q}}>k\})\leq\varepsilon$. Utilizing
the semi-continuity of ${\tau}_{{Q}}$ on $(O_{1})^{c},$ we deduce that
$(O_{1})^{c}\cap\{{\tau}_{{Q}}\leq k\}$ is a closed set, and thus,
$O:=O_{1}\cup\{{\tau}_{{Q}}>k\}$ is an open set. It is easy to see that
$c(O)\leq2\varepsilon$, and ${\tau}_{{Q}}$ and
${\tau}_{\overline{Q}}$ are continuous on $O^{c}$.

(ii) Note that
\[
c({\tau}_{{Q}}>k)=\hat{\mathbb{E}}[{1\cdot}I_{\{{\tau}_{{Q}}>k\}}]\leq
\hat{\mathbb{E}}[{\tau}_{{Q}}I_{\{{\tau}_{{Q}}>k\}}]\rightarrow0,\ \ \ \text{as}
\ 1\leq k\rightarrow\infty.
\]
Then ${\tau}_{{Q}}$ is quasi-continuous and the conclusion now follows
directly from the characterization theorem of $L_{C}^{1}(\Omega)$ (Theorem
\ref{LG-ch}).
\end{proof}

Obviously, (\ref{Myeq3.3}) implies (\ref{Myeq3.13}). Now we provide a sufficient condition for (\ref{Myeq3.3}).

\begin{proposition}
\label{Myth3.3} Let $Q$ be a bounded open set and $Y$ be a $\mathcal{P}%
$-semimartingale. Assume that, for some $1\leq l\leq d,$ there exist some
constants $\varepsilon>0\ $and $\lambda\neq0$ such that
\[
\lambda dA_{t}^{P,l}+d\langle M^{P,l}\rangle_{t}\geq\varepsilon dt \ \ \text{ on
}[0,{{\tau}_{\overline{Q}}}],\text{ }P\text{-a.s., for each }P\in\mathcal{{P}%
}\text{,}%
\]
where $M^{P,l}$ and $A^{P,l}$ are the $l$-th components of $M^{P}\ $and
$A^{P},$ respectively. Then there exists a constant $C>0$ depending only on
$\lambda,\varepsilon$ and the diameter of ${Q}$ such that,
\begin{equation}
\hat{\mathbb{E}}[({{\tau}_{\overline{Q}}})^{2}]\leq C.
\end{equation}

\end{proposition}

\begin{proof}
We mainly use an auxiliary function from \cite[p. 145]{Fr}. Without loss of
generality, we can assume $0\in Q$ and $l=1$.

\textit{Step 1.} Let $P\in\mathcal{{P}}$\ be given. Let $h(y):=\beta e^{\frac
{2y_{1}}{\lambda}},$ and take $\beta>0$ large enough such that $P\text{-a.s.}$
for each $y\in\overline{Q}$,
\[
\frac{2}{\lambda}h(y)(dA_{t}^{P,1}+\frac{1}{\lambda}d\langle M^{P,1}%
\rangle_{t})=\frac{2}{\lambda^{2}}h(y)(\lambda dA_{t}^{P,1}+d\langle
M^{P,1}\rangle_{t})\geq dt\ \ \text{ on }[0,{{\tau}_{\overline{Q}}}].
\]
By It\^{o}'s formula, we have
\begin{align*}
h(Y_{{{\tau}_{\overline{Q}}}\wedge t})-h(Y_{0})= &  \int_{0}^{{{\tau
}_{\overline{Q}}}\wedge t}\frac{2}{\lambda}h(Y_{s})dM_{s}^{P,1}+\int%
_{0}^{{{\tau}_{\overline{Q}}}\wedge t}\frac{2}{\lambda}h(Y_{s})dA_{s}%
^{P,1}+\frac{1}{2}\int_{0}^{{{\tau}_{\overline{Q}}}\wedge t}\frac{4}%
{\lambda^{2}}h(Y_{s})d\langle M^{P,1}\rangle_{s}\\
= &  \int_{0}^{{{\tau}_{\overline{Q}}}\wedge t}\frac{2}{\lambda}h(Y_{s}%
)dM_{s}^{P,1}+\int_{0}^{{{\tau}_{\overline{Q}}}\wedge t}\frac{2}{\lambda^{2}%
}h(Y_{s})(\lambda dA_{s}^{P,1}+d\langle M^{P,1}\rangle_{s}).
\end{align*}
Taking expectation on both sides and using a standard localization argument when necessary, we get
\[
2C_{h}\geq{{E}}_{P}[{{\tau}_{\overline{Q}}}\wedge t].
\]
where $C_{h}$ is the bound of $h$ on $\overline{Q},$ which is independent of
$P\in\mathcal{{P}}$ and $t$.

\textit{Step 2.} Consider $th(y)$, where $h$ with $\beta$ is given as in Step 1.
Applying It\^{o}'s formula, we have
\begin{align*}
({{\tau}_{\overline{Q}}}\wedge t)h(Y_{{{\tau}_{\overline{Q}}\wedge t}})=  &
\int_{0}^{{{{{\tau}_{\overline{Q}}}}}\wedge t}h(Y_{s})ds+\int_{0}^{{{{{\tau
}_{\overline{Q}}}}}\wedge t}\frac{2s}{\lambda}h(Y_{s})dM_{s}^{P,1}+\int%
_{0}^{{{\tau}_{\overline{Q}}}\wedge t}\frac{2s}{\lambda}h(Y_{s})dA_{s}^{P,1}\\
&  +\frac{1}{2}\int_{0}^{{{{{\tau}_{\overline{Q}}}}}\wedge t}\frac{4s}
{\lambda^{2}}h(Y_{s})d\langle M^{P,1}\rangle_{s}\\
\geq &  \int_{0}^{{{{{\tau}_{\overline{Q}}}}}\wedge t}\frac{2s}{\lambda
}h(Y_{s})dM_{s}^{P,1}+\int_{0}^{{{\tau}_{\overline{Q}}}\wedge t}\frac
{2s}{\lambda^{2}}h(Y_{s})(\lambda dA_{s}^{P,1}+d\langle M^{P,1}\rangle_{s}).
\end{align*}
Taking expectation on both sides, we get
\[
C_{h}{{E}}_{P}[{{\tau}_{\overline{Q}}}\wedge t]\geq{{E}}_{P}[({{\tau
}_{\overline{Q}}}\wedge t)h(Y_{{{\tau}_{\overline{Q}}}})]\geq{{E}}_{P}
[\int_{0}^{{{\tau}_{\overline{Q}}}\wedge t}sds]=\frac{1}{2}{{E}}_{P}[({{\tau
}_{\overline{Q}}\wedge t})^{2}],
\]
which together with Step 1 implies
\[
{{E}}_{P}[({{\tau}_{\overline{Q}}\wedge t})^{2}]\leq4(C_{h})^{2}.
\]
Taking supremum over $P\in\mathcal{{P}}$ and then letting $t\rightarrow
\infty,$ we obtain
\[
\hat{\mathbb{E}}[({{\tau}_{\overline{Q}}})^{2}]\leq4(C_{h})^{2},
\]
as desired.
\end{proof}

\begin{remark}
\upshape{
	If $\hat{\mathbb{E}}[({{\tau}_{\overline{Q}}})^{2}]<\infty$, then by the
	Markov inequality, we obtain that (ii) in Theorem \ref{Myth3} holds: $\hat{\mathbb{E}}[{{\tau}_{\overline{Q}}}I_{\{{{\tau
			}_{\overline{Q}}}>k \}}]\leq\frac{\hat{\mathbb{E}}[({{\tau}_{\overline{Q}}		})^{2}]}{k}\rightarrow0,\ \text{as}\ k\rightarrow\infty.$}
\end{remark}

\section{Quasi-continuous processes}

In the previous section, the regularity theorem for exit times (Theorem
\ref{Myth3.5}) was established under the assumption that the $\mathcal{P}%
$-semimartingale $Y$ has some kind of regularity which is called
quasi-continuity in the process sense. In the present section, we shall give
a characterization theorem on the quasi-continuity of processes as well as
some related properties of stopped processes.

\subsection{Characterization of quasi-continuous processes}

Assume that $\mathcal{P}$ is a family of probability measures on $\Omega$,
$c$ and $\mathbb{\hat{E}}$ are the corresponding upper capacity and
expectation, respectively.

Now we give a general criterion (characterization) on the quasi-continuity of
processes. It is convenient to first introduce the notion of quasi-continuity
on the finite interval. We say that a process $F=(F_{t})_{t\in[0,\infty)}$ is
quasi-continuous on $\Omega\times[0,T]$ if for each $\varepsilon>0$, there
exists an open set $G\subset\Omega$ with $c(G)<\varepsilon$ such that
$F_{\cdot}(\cdot)$ is continuous on $G^{c}\times\lbrack0,T]$. Obviously, if
$F$ is quasi-continuous on $\Omega\times[0,\infty)$, then $F$ is
quasi-continuous on $\Omega\times[0,T]$, for each $T>0$.

\begin{theorem}
\label{Myth3.11} Let $X:\Omega\times[0,\infty)\rightarrow\mathbb{R}$ be a stochastic process, i.e., $X_t$ is $\mathcal{B}(\Omega)$-measurable for each $t\geq0$.

\begin{description}
	\item[\rm{{(i)}}] $X$ has a quasi-continuous version on $\Omega
\times\lbrack0,T]$ if and only if  we can find a sequence $X^{n}\in
C(\Omega\times\lbrack0,T])$ such that, for each $\varepsilon>0,$
\begin{equation}
c(\{\sup_{0\leq t\leq T}|X_{t}^{n}-X_{t}|>\varepsilon\})\rightarrow
0,\ \ \ \text{as}\ n\rightarrow\infty\text{.} \label{Myeq3.5}%
\end{equation}
Moreover, $X$ and this version are q.s. continuous in $t\in\lbrack0,T],$
i.e., continuous in $t\in\lbrack0,T]$ for q.s. $\omega\in\Omega.$

\item[\rm{{(ii)}}] $X$ has a quasi-continuous version on $\Omega
\times\lbrack0,\infty)$ if and only if for each $T>0,$ there exists a sequence
$X^{n}\in C(\Omega\times\lbrack0,T])$ such that (\ref{Myeq3.5}) holds. Also,
$X$ and this version are  q.s. continuous in $t\in\lbrack0,\infty).$
\end{description}
\end{theorem}

\begin{proof}
(i) Note that
\[
c(\{\sup_{0\leq t\leq T}|X_{t}^{n}-X_{t}^{m}|>\varepsilon\})\rightarrow
0,\ \ \ \text{as}\ n,m\rightarrow\infty.
\]
We can find a sequence $(X_{t}^{n_{k}})_{k\geq1}$ such that
\[
c(\{\sup_{0\leq t\leq T}|X_{t}^{n_{k+1}}-X_{t}^{n_{k}}|>\frac{1}{2^{k}}
\})\leq\frac{1}{2^{k}},\ \ \ \forall k\geq1.
\]
Denote
\[
A_{k}=\{\sup_{0\leq t\leq T}|X_{t}^{n_{k+1}}-X_{t}^{n_{k}}|>\frac{1}{2^{k}%
}\}.
\]
Then
\[
\sum_{k=1}^{\infty}c(A_{k})\leq\sum_{k=1}^{\infty}\frac{1}{2^{k}}=1.
\]
As a consequence, by Borel-Cantelli Lemma, $D_{T}:=\limsup_{k\rightarrow
\infty}A_{k}$ is polar. Since $X^{n_{k}}_t$ is continuous on $\Omega$, for each $t$ and $k\geq1$, then $\sup_{0\leq t\leq T}
|X_{t}^{n_{k+1}}-X_{t}^{n_{k}}|$ is lower semi-continuous on $\Omega,$ and thus, the set
$A_{k}$ is open on $\Omega$ (The above semi-continuity is from a lemma that the supremum of a family of (lower semi-) continuous functions is still lower semi-continuous. Moreover,  $\sup_{0\leq t\leq T}
|X_{t}^{n_{k+1}}-X_{t}^{n_{k}}|$ can be shown to be continuous on $\Omega$ by a deep result in the classical analysis (see \cite{WB}), by the fact that $X^{n_{k}}_t$ is continuous on $\Omega
\times\lbrack0,T]$, for each $k\geq1$,  and $\lbrack0,T]$ is compact. But the lower semi-continuous is enough for us to guarantee that $A_{k}$ is open). Therefore, $\cup_{k\geq k_{0}}A_{k}\supset D_{T}$
is an open set and can have any sufficient small capacity when $k_{0}$ large
enough. We define the limit of $X^{n_{k}}$ on $[0,T]$ by
\[
I^T_{t}(\omega)=\limsup_{k\rightarrow\infty}X_{t}^{n_{k}}(\omega).
\]
As each $X^{n_{k}}$ is continuous in $(\omega,t)$, for all $k\geq1$, and
$X^{n_{k}}$ converges uniformly on $(\cup_{k\geq k_{0}}A_{k})^{c}\times
\lbrack0,T].$ Therefore, $I_{\cdot}^{T}(\cdot)$ is continuous on $(\cup_{k\geq
k_{0}}A_{k})^{c}\times\lbrack0,T].$ Thus, the process $I^{T}$ is
quasi-continuous on $\Omega\times\lbrack0,T]$.

Moreover, note that for each $\omega\in(D_{T})^{c},$ $t\rightarrow
X_{t}^{n_{k}}(\omega)$ converges to $t\rightarrow I_{t}^{T}(\omega)$
uniformly, thus q.s. $t\rightarrow I_{t}^{T}(\omega)$ is continuous on $[0,T].$ So is $X$ since they are versions of each other on $[0,T].$

To prove the reverse direction, we can assume that $X$ itself is quasi-continuous on
$\Omega\times\lbrack0,T]$, since if $X'$ is the quasi-continuous version of $X$ on
$\Omega\times\lbrack0,T]$, then $\sup_{0\leq t\leq T}|X_t-X'_t|=0$ q.s.  For any $\varepsilon>0$, we can find an open set
$G\subset\Omega$ with $c(G)<\varepsilon$ such that $X_{\cdot}(\cdot)$ is
continuous on $G^{c}\times\lbrack0,T]$. By the Tietze's extension theorem,
there exists a $Y$ which is continuous on $\Omega\times\lbrack0,T]$ such that
$X=Y$ on $G^{c}\times\lbrack0,T].$ Then
\[
c(\{\sup_{0\leq t\leq T}|Y_{t}-X_{t}|>\varepsilon\})\leq c(\{\sup_{0\leq t\leq T}|Y_{t}-X_{t}|>\varepsilon\}\cap G)\leq
c(G)\leq\varepsilon.
\]

(ii) For each $k\geq1,$ from (i), we get a quasi-continuous version $I^{k}$ of $X$ on $[0,k]$. Then $I_{t}
^{k}=I_{t}^{k^{\prime}},0\leq t\leq k\wedge k^{\prime}$ q.s. for each
$k,k^{\prime}\geq1.$ Denote the polar sets
\[
F^{k,k^{\prime}}:=\{\omega\in\Omega:I_{t}^{k}(\omega)=I_{t}^{k^{\prime}}
(\omega),0\leq t\leq k\wedge k^{\prime}\text{ does not hold}\}\ \ \ \text{and
\ }F:=\cup_{k,k^{\prime}\geq1}F^{k,k^{\prime}}.
\]
Then we can define
\[
I_{t}(\omega)=
\begin{cases}
I_{t}^{k}(\omega),\ t\leq k;\quad\omega\in F^{c},\\
0;\quad\quad\quad\ \ \ \ \ \ \ \ \ \omega\in F.
\end{cases}
\]
For any given $\varepsilon>0$ and for each $k\geq1,$ from (i), we can find an
open set $O_{k}$ such that $c(O_{k})\leq\frac{\varepsilon}{2^{k}}$ and $I^{k}$
is continuous on $(O_{k})^{c}\times\lbrack0,k].$ Denoting the open set
$O^{\prime}=\cup_{k\geq1}O_{k},$ then $c(O^{\prime})\leq\varepsilon.$ We also
denote $O=O^{\prime}\cup F$ and obviously that $c(O)\leq\varepsilon.$ It is
easy to see that $I$ is continuous on $O^{c}\times\lbrack0,\infty).$ Now it
remains to prove that $O$ is open. To that end, it suffices to show that
$O^{c}$ is closed. Given any $k,k^{\prime}\geq1$. For every $t\in
\lbrack0,k\wedge k^{\prime}]$, since $I_{t}^{k},I_{t}^{k^{\prime}}$ is
continuous on $(O^{\prime})^{c},$ then $\{\omega\in\Omega:I_{t}^{k}%
(\omega)=I_{t} ^{k^{\prime}}(\omega)\}\cap(O^{\prime})^{c}=\{\omega\in
\Omega:I_{t}^{k} (\omega)-I_{t}^{k^{\prime}}(\omega)=0\}\cap(O^{\prime})^{c}$
is a closed set. Thus,
\[
(F^{k,k^{\prime}})^{c}\cap(O^{\prime})^{c}=(\cap_{t\in\lbrack0,k\wedge
k^{\prime}]}\{\omega\in\Omega:I_{t}^{k}(\omega)=I_{t}^{k^{\prime}}
(\omega)\})\cap(O^{\prime})^{c}=\cap_{t\in\lbrack0,k\wedge k^{\prime}
]}(\{\omega\in\Omega:I_{t}^{k}(\omega)=I_{t}^{k^{\prime}}(\omega
)\}\cap(O^{\prime})^{c})
\]
is closed. This implies
\[
O^{c}=(O^{\prime})^{c}\cap F^{c}=(O^{\prime})^{c}\cap(\cap_{k,k^{\prime}\geq
1}(F^{k,k^{\prime}})^{c})=\cap_{k,k^{\prime}\geq1}((F^{k,k^{\prime}})^{c}
\cap(O^{\prime})^{c})
\]
is closed, as desired.

The q.s. continuity of $I$ in $t$ on $[0,\infty)$
follows from the above definition of $I$ and the q.s.  continuity of $I^k$ in $t$ on $[0,k]$ for each $k\geq 1$. Moreover, $X$ is also q.s. continuous since it is a version of $I$ on $[0,\infty)$.

Now we prove the reverse direction. If $X$ is quasi-continuous on
$\Omega\times\lbrack0,\infty)$, then $X$ is quasi-continuous on $\Omega
\times\lbrack0,T],$ for each $T>0,$ and the conclusion follows from (i).
\end{proof}

\begin{remark}
	\upshape{
	From the proof of the above theorem, we know that the direction of obtaining the approximation property in the form of (\ref{Myeq3.5}) from the quasi-continuity is always trivial. In fact,  we can get a better approximation property in (ii): If $X$ is quasi-continuous on $\Omega\times\lbrack0,\infty)$, then there exists a sequence $X^{n}\in
	C(\Omega\times\lbrack0,\infty))$ such that, for each $\varepsilon>0,$
	\begin{equation}\label{4-22}
	c(\{\sup_{0\leq t< \infty}|X_{t}^{n}-X_{t}|>\varepsilon\})\rightarrow
	0,\ \ \ \text{as}\ n\rightarrow\infty\text{.}
	\end{equation}
	Indeed, for any $\varepsilon>0$, by a similar analysis as in the proof of (i), we can find an open set
	$G\subset\Omega$ with $c(G)<\varepsilon$ and a function  $Y$ which is continuous on $\Omega\times\lbrack0,\infty)$ such that
	$X=Y$ on $G^{c}\times\lbrack0,\infty).$ Then
	\[
	c(\{\sup_{0\leq t<\infty }|Y_{t}-X_{t}|>\varepsilon\})\leq c(\{\sup_{0\leq t<\infty}|Y_{t}-X_{t}|>\varepsilon\}\cap G)\leq
	c(G)\leq\varepsilon.
	\]
	So we can see that the weak form of approximation condition in (ii) of Theorem \ref{Myth3.11} that for each $T>0,$ (\ref{Myeq3.5}) holds for some sequence $X^n$, is equivalent to the stronger form (\ref{4-22}), which does not seem obvious. For Theorem \ref{Myth3.11}, we mainly use its nontrivial direction of obtaining the quasi-continuity from the approximation property in the concrete problems.
So in (ii), we prefer the weak form condition since, by the observation that it is verified on finite time interval with  approximation sequence $X^n$ that can depend on $T$, it is more convenient to apply.	}
\end{remark}

In particular, taking $T=0$ in Theorem \ref{Myth3.11} (i), we get the
corresponding quasi-continuity characterization theorem for random variables,
which also generalizes Theorem \ref{LG-ch} a bit.

\begin{corollary}
Let $X:\Omega\rightarrow\mathbb{R}$ be a random variable. Then $X$ is
quasi-continuous if and only if there exists a sequence $X^{n}\in C(\Omega)$
such that, for each $\varepsilon>0$,
\begin{equation}
c(\{|X^{n}-X|>\varepsilon\})\rightarrow0,\ \ \ \text{as}\ n\rightarrow\infty.
\end{equation}

\end{corollary}

The following two results concern the quasi-continuity of stopped processes.

\begin{proposition}
\label{Myth3.10}Let $X=(X_{t})_{t\in[0,\infty)}$ be a process. The random
variable $X_{\tau}$ is quasi-continuous if
 $X$ is a quasi-continuous process on $\Omega\times[0,T]$ $($$\Omega
\times[0,\infty)$ resp.$)$
and $\tau:\Omega\rightarrow[0,T]$ $($$[0,\infty)$ resp.$)$ is a quasi-continuous stopping time.
\end{proposition}

\begin{proof}
We just prove the conclusion on $[0,T]$, and the proof for the other
part is similar. For any $\varepsilon>0$, we can find an open set
$G\subset\Omega$ such that $c(G^{c})\leq\varepsilon$ and on $G^{c}%
\times\lbrack0,T]$, $X$ is continuous. Moreover, we can also find an open set
$O\subset\Omega$ such that $\tau$ is continuous on $O^{c}$. Then on $G^{c}\cap
O^{c}=(G\cup O)^{c}$, it is easy to see that $X_{\tau}$ is continuous.
\end{proof}

\begin{proposition}\label{Myth4.3}
Let $X=(X_{t})_{t\in[0,\infty)}$ be a process. Then
 the process $(X_{\tau\wedge t})_{t\in[0,T]}$ $($$(X_{\tau\wedge t})_{t\in[0,\infty)}$ resp.$)$ is
quasi-continuous on $\Omega\times[0,T]$ $($$\Omega\times[0,\infty)$ resp.$)$ if $X$ is and $\tau$ is a
quasi-continuous stopping time.
\end{proposition}

\begin{proof}
The proof is similar to that of Proposition \ref{Myth3.10}, so we omit it.
\end{proof}

\begin{remark}
\upshape{
We remark that Proposition \ref{Myth3.10} is a special case of Proposition
\ref{Myth4.3} from Remark \ref{Re2.1}. But it should be beneficial to give
Proposition \ref{Myth3.10} explicitly as above due to its potential broader use.}
\end{remark}

\subsection{Application to $G$-expectation space}

For any given family of probability measures, the canonical process $B$ is
continuous in $(\omega,t)$, and thus is trivially quasi-continuous. Now we shall use Theorem \ref{Myth3.11} to obtain some non-trivial
quasi-continuous processes in the case that $\mathcal{P}$ is the $G$-expectation family, i.e., the upper expectation of $\mathcal{P}$ is a $G$-expectation. Let us first
briefly review the construction of $G$-expectation, and more details can be
found in \cite{D-H-P,Peng 1}.

Let $\Gamma$ be a bounded and closed subset of $\mathbb{S}_{+}(k)$, where $\mathbb{S}_{+}(k)$ is the collection of nonnegative $k\times k$ symmetric
matrices.
The $G$-expectation $\mathbb{\hat{E}}$ is the upper expectation of the
probability family%
\[
\mathcal{P=}\left\{  P:P\text{ is a probability measure on}\ \Omega
\ \text{such that }B\text{ is a martingale and }\frac{d\langle B\rangle
_{t}^{P}}{dt}\in\Gamma\right\}  ,
\]
under which the canonical process $B$ is called $G$-Brownian motion. In the
$G$-expectation case,$\ L_{C}^{1}(\Omega)$ is usually denoted by $L_{G}%
^{1}(\Omega)$ and the conditional $G$-expectation $\mathbb{\hat{E}}_{t}%
[\cdot]$ is well-defined on $L_{G}^{1}(\Omega).$

An adapted process $(M_{t})_{t\geq0}$ is called a $G$-martingale if for each
$s\leq t,$ $M_{t}\in L_{G}^{1}(\Omega_{t})$ and $\hat{\mathbb{E}}_{s}%
[M_{t}]=M_{s}$, where $\Omega_{t}=\{\omega_{\cdot\wedge t}:\omega\in\Omega\}$
and $L_{G}^{1}(\Omega_{t})$ is defined similar to $L_{G}^{1}(\Omega)$ with
$\Omega$ replaced by $\Omega_{t}$. Furthermore, a $G$-martingale $M$ is called
symmetric if $-M$ is also a $G$-martingale. We remark that, if $M$ is a
symmetric $G$-martingale, then it is a $\mathcal{P}$-martingale, i.e., it is a
martingale under each $P\in\mathcal{P}$. In general, a $G$-martingale is a
$\mathcal{P}$-supermartingale, see \cite{PZ,LPS} for more discussions.

Let $M_{G}^{0}(0,T)$ be the collection of processes in the form: for a given
partition $\{t_{0},\cdot\cdot\cdot,t_{N}\}$ of $[0,T]$,
\[
\eta_{t}(\omega)=\sum_{j=0}^{N-1}\xi_{j}(\omega)I_{[t_{j},t_{j+1})}(t),
\]
where $\xi_{j}\in C_{b}( \Omega_{t_{j}})$, $j=0,1,2,\cdot\cdot\cdot,N-1$. For
$p\geq1$ and $\eta\in M_{G}^{0}(0,T)$, let $\Vert\eta\Vert_{M_{G}^{p}%
}=\{\mathbb{\hat{E}}[\int_{0}^{T}|\eta_{s}|^{p}ds]\}^{1/p},$ $\Vert\eta
\Vert_{H_{G}^{p}}=\{\mathbb{\hat{E}}[(\int_{0}^{T}|\eta_{s}|^{2}%
ds)^{p/2}]\}^{1/p}$, and denote by $M_{G}^{p}(0,T)$ and $H_{G}^{p}(0,T)$ the
completions of $M_{G}^{0}(0,T)$ under the norms $\Vert\cdot\Vert_{M_{G}^{p}}$ and
$\Vert\cdot\Vert_{H_{G}^{p}}$, respectively. For each $1\leq i,j\leq k$, we denote by $\langle
B^{i},B^{j}\rangle$  the cross-variation process of $B$. Then for two processes $\eta\in
H_{G}^{p}(0,T)$ and $\xi\in M_{G}^{p}(0,T)$, the $G$-It\^{o} integrals
$\int_{0}^{t}\eta_{s}dB_{s}^{i}$ and $\int_{0}^{t}\xi_{s}d\langle B^{i}%
,B^{j}\rangle_{s}$, $\int_{0}^{t}\xi_{s}ds$ are well-defined, and $\int%
_{0}^{t}\eta_{s}dB_{s}^{i}$ is a symmetric $G$-martingale.

In the following of this subsection, we always assume that $\mathcal{P}$ is a
family of probability measures corresponding to $G$-expectation as above.

Theorem \ref{Myth3.11} contains the following three typical processes in the
$G$-expectation space.

\begin{proposition}
\label{Myth4.5} We have:

\begin{description}
\item[\rm{(i)}] $G$-martingale $M$ has a quasi-continuous modification on
$\Omega\times[0,\infty)$.

\item[\rm{(ii)}] If $\eta\in M_{G}^{1}(0,T)$ $(\cap_{T>0}M_{G}^{1}(0,T)$
resp.$)$, then the process $A_{t}:=\int_{0}^{t}\eta_{s}ds$ has a
quasi-continuous modification on $\Omega\times[0,T]$ $(\Omega\times[0,\infty)$
resp.$)$.

\item[\rm{(iii)}] If $\eta\in M_{G}^{1}(0,T)$ $(\cap_{T>0}M_{G}^{1}(0,T)$
resp.$)$, then the process $A_{t}:=\int_{0}^{t}\eta_{s}d\langle B^{i}%
,B^{j}\rangle_{s}$ has a quasi-continuous modification on $\Omega\times[0,T]$
$(\Omega\times[0,\infty)$ resp.$)$.
\end{description}
\end{proposition}

\begin{proof}
(i). For each $T$, since $M_{T}\in L_{G}^{1}\left(  \Omega_{T}\right),$ according to  \cite{D-H-P}, we
can find $\xi^{n}\in L_{ip}(\Omega_{T})$ such that $\xi^{n}\rightarrow M_{T}$
under the norm $\mathbb{\hat{E}}[|\cdot|]$, where%
\[
L_{ip}(\Omega_{T}):=\{\varphi(B_{t_{1}},B_{t_{2}}-B_{t_{1}}\cdots,B_{t_{n}%
}-B_{t_{n-1}}):n\in\mathbb{N},0\leq t_{1}<t_{2}\cdots<t_{n}\leq T,\varphi\in
C_{b.Lip}(\mathbb{R}^{k\times n})\}.
\]
From the definition of conditional $G$-expectation (c.f. Chapter III of
\cite{Peng 1}), we can see that the process $\mathbb{\hat{E}}_{t}[\xi
^{n}]$ is continuous on $\Omega\times[0,T].$ By the following Lemma \ref{Le4-9}, we can take the continuous modifications of $G$-martingales $M_t=\mathbb{\hat{E}%
}_{t}[M_{T}]$ and $\mathbb{\hat{E}}_{t}[|\xi^{n}-M_{T}%
|]$. Since for any given
$P\in\mathcal{P}$, $\mathbb{\hat{E}}_{t}[|\xi^{n}-M_{T}|]$ is a
supermartingale, we can apply the Doob's martingale
inequality (see, e.g., Theorem 2.42 of \cite{HWY}) to obtain that for each $\varepsilon>0,$
\begin{align*}
P(\{\sup_{0\leq t\leq T}|\mathbb{\hat{E}}_{t}[\xi^{n}]-M_{t}| >\varepsilon\})
&  =P(\{\sup_{0\leq t\leq T}|\mathbb{\hat{E}}_{t}[\xi^{n}]-\mathbb{\hat{E}%
}_{t}[M_{T}]|>\varepsilon\})\\
&  \leq P(\{\sup_{0\leq t\leq T}\mathbb{\hat{E}}_{t}[|\xi^{n}-M_{T}|]>\varepsilon\})\\
&  \leq\frac{1}{\varepsilon}\mathbb{\hat{E}}[|\xi^{n}-M_{T}|].
\end{align*}
Taking supremum over $P\in\mathcal{P}$, we obtain
\[
c(\{\sup_{0\leq t\leq T}|\mathbb{\hat{E}}_{t}[\xi^{n}]-M_{t}|>\varepsilon
\})\leq\frac{1}{\varepsilon}\mathbb{\hat{E}}[|\xi^{n}-M_{T}|]\rightarrow
0,\ \ \ \text{as}\ n\rightarrow\infty.
\]
Now applying Theorem \ref{Myth3.11}, we deduce that $M$ is quasi-continuous.

(ii). We can find a sequence $\eta^{n}\in M_{G}^{0}(0,T)$ such that $\eta
^{n}\rightarrow\eta$ in $M_{G}^{1}(0,T)$. Then the conclusion follows from the
observation that the process $(\int_{0}^{t}\eta^{n}_{s}ds)_{t\geq0}$ is
continuous on $\Omega\times[0,T]$ and
\[
\mathbb{\hat{E}}[\sup_{0\leq t\leq T}|\int_{0}^{t}\eta_{s}^{n}ds-\int_{0}%
^{t}\eta_{s}ds|]\leq\mathbb{\hat{E}}[|\int_{0}^{T}|\eta_{s}^{n}-\eta
_{s}|ds]\rightarrow0,\ \ \text{ as }n\rightarrow\infty.
\]

(iii). Note that $M_{t}:=\int_{0}^{t}\eta_{s}d\langle B^{i},B^{j}\rangle
_{s}-\int_{0}^{t}2G(\tilde{\eta}_{s})ds$ is a $G$-martingale (see Chapter IV
of \cite{Peng 1}), where $\tilde{\eta}=(\tilde{\eta}^{ml})_{m,l=1}^{k}$ is
defined by
\[
\tilde{\eta}_{t}^{ml}=%
\begin{cases}
\eta_{t};\quad m=i\ \text{and}\ l=j,\\
0;\ \ \ \ \text{otherwise}.
\end{cases}
\]
Then we deduce the result from (i) and (ii).
\end{proof}

\begin{remark}
\upshape{
We remark that the result (i) on finite interval $[0,T]$ has already been
obtained in \cite{Song1}. Compared with this, our proof is simple and
different, and moreover, does not rely on the non-degeneracy assumption on
$\Gamma$.}
\end{remark}

In the above proof, the following continuity modification theorem for  $G$-martingales $M$ is needed. It corresponds to the classical fact in the linear stochastic analysis that every martingale for Brownian motion filtration has a continuous modification, and partial result under the additional assumption that $M_T\in L^p_G(\Omega)$ for $T>0$, for some $p>1$, on this direction has already been proved as a byproduct in the $G$-martingale representation theorem (see \cite{STZ,Song1}).
\begin{lemma}\label{Le4-9}
	Any $G$-martingale $M$ has a continuous modification.
\end{lemma}
\begin{proof}
We employ the notation in the proof of Proposition \ref{Myth4.5} (i). For any given $T>0$ and $M_{T}\in L_{G}^{1}\left(  \Omega_{T}\right)$, there exists some $\xi^{n}\in L_{ip}(\Omega_{T})$ such that $\xi^{n}\rightarrow M_{T}$
under the norm $\mathbb{\hat{E}}[|\cdot|]$. Then by the definition of conditional $G$-expectation, $t\rightarrow\mathbb{\hat{E}}_{t}[\xi
^{n}]$ and $t\rightarrow\mathbb{\hat{E}}_{t}[|\xi
^{n}-\xi
^{m}|]$ are continuous on $[0,T],$ for each $\omega\in\Omega$, and by a similar calculation as in the proof of Proposition \ref{Myth4.5} (i), we have
	\[
	c(\{\sup_{0\leq t\leq T}|\mathbb{\hat{E}}_{t}[\xi
	^{n}]-\mathbb{\hat{E}}_{t}[\xi
	^{m}]|>\varepsilon\})\rightarrow
	0,\ \ \ \text{as}\ n,m\rightarrow\infty.
	\]
	Now from a similar analysis as in the proof of Theorem \ref{Myth3.11} (i), we can extract a q.s uniformly convergent subsequence $\mathbb{\hat{E}}_{t}[\xi
	^{n_k}]$ such that $t\rightarrow\limsup_{k\rightarrow\infty}\mathbb{\hat{E}}_{t}[\xi
	^{n_k}]$ is continuous and it is a modification of $M$.
\end{proof}

A $G$-martingale stopped at a quasi-continuous stopping time is still a $G$-martingale.

\begin{corollary}\label{Myth4-10}
Let $\tau$ be a quasi-continuous stopping time. If $(M_{t})_{t\geq0}$ is a
$G$-martingale $($symmetric $G$-martingale resp.$)$, then $(M_{t\wedge\tau
})_{t\geq0}$ is still a $G$-martingale $($symmetric $G$-martingale resp.$)$.
\end{corollary}

\begin{proof}
We just prove the $G$-martingale case, from which the symmetric case follows
by applying the conclusion to $M$ and $-M$.

For any $t$ and stopping time $\sigma\leq t,$ let $\mathbb{\hat{E}}_{\sigma}$
be the conditional $G$-expectation at $\sigma$ as defined in \cite{NH,HJL}. By
the optional sampling theorem for $G$-martingales (see \cite{NH}), we have
\begin{equation}
\mathbb{\hat{E}}_{\sigma}[M_{t}]=M_{\sigma}. \label{Myeq4.6}%
\end{equation}

From Proposition \ref{Myth3.10}, the random variable $M_{t\wedge\tau}$ is
quasi-continuous. Moreover, note that  from (\ref{Myeq4.6}) and the properties
of conditional $G$-expectation,
\[
c(\{|M_{t\wedge\tau}|>N\})\leq\frac{\mathbb{\hat{E}}[|M_{t\wedge\tau}|]}%
{N}=\frac{\mathbb{\hat{E}}[|\mathbb{\hat{E}}_{t\wedge\tau}\mathbb{[}M_{t}%
]|]}{N}\leq\frac{\mathbb{\hat{E}}[\mathbb{\hat{E}}_{t\wedge\tau}[|M_{t}%
|]]}{N}=\frac{\mathbb{\hat{E}}[|M_{t}|]}{N}\rightarrow0,\ \ \text{ as
}N\rightarrow\infty.
\]
Then applying Proposition 19 in \cite{D-H-P} yields that
\begin{align*}
\mathbb{\hat{E}}[|M_{t\wedge\tau}|I_{\{|M_{t\wedge\tau}|>N\}}]  &
=\mathbb{\hat{E}}[|\mathbb{\hat{E}}_{t\wedge\tau}\mathbb{[}M_{t}%
]|I_{\{|M_{t\wedge\tau}|>N\}}]\\
&  \leq\mathbb{\hat{E}}[\mathbb{\hat{E}}_{t\wedge\tau}[|M_{t}|]I_{\{|M_{t\wedge\tau}|>N\}}]\\
&  =\mathbb{\hat{E}}[|M_{t}|I_{\{|M_{t\wedge\tau}|>N\}}]\\
&  \rightarrow0, \ \ \text{ as }N\rightarrow\infty.
\end{align*}
Therefore, by the characterization Theorem \ref{LG-ch}, we deduce that
$M_{t\wedge\tau}\in L_{G}^{1}\left(  \Omega_{t}\right)  .$

Now it remains to show the martingale property. Indeed, from (\ref{Myeq4.6})
and the properties of conditional $G$-expectation, for each $s\geq t,$ we
have
\begin{align*}
\hat{\mathbb{E}}_{t}[M_{s\wedge\tau}]  &  =\hat{\mathbb{E}}_{t}[M_{s\wedge
\tau}I_{\{\tau\geq t\}}]+\hat{\mathbb{E}}_{t}[M_{s\wedge\tau}I_{\{\tau<t\}}]\\
&  =\hat{\mathbb{E}}_{t}[M_{(s\wedge\tau)\vee t}]I_{\{\tau\geq t\}}%
+\hat{\mathbb{E}}_{t}[M_{\tau\wedge t}]I_{\{\tau<t\}}\\
&  =\hat{\mathbb{E}}_{t}[\hat{\mathbb{E}}_{_{(s\wedge\tau)\vee t}}%
[M_{s}]]I_{\{s\wedge\tau\geq t\}}+M_{\tau\wedge t}I_{\{\tau<t\}}\\
&  =\hat{\mathbb{E}}_{t} [M_{s}]I_{\{s\wedge\tau\geq t\}}+M_{\tau\wedge
t}I_{\{\tau<t\}}\\
&  =M_{t}I_{\{s\wedge\tau\geq t\}}+M_{\tau}I_{\{s\wedge\tau<t\}}\\
&  =M_{\tau\wedge t}.
\end{align*}
This completes the proof.
\end{proof}

We close this section with a regularity theorem for the stopping of stochastic integrals.

\begin{proposition}
\label{Myth3.6}Let $\tau\leq T$ be a quasi-continuous stopping time. Then for
each $p\geq1$, we have
\begin{equation}
I_{[0,\tau]}\in M_{G}^{p}(0,T).
\end{equation}

\end{proposition}

\begin{proof}
Without loss of generality, we assume that $\tau\leq T$. For each
$k\in\mathbb{N}$, by the partition of unit theorem, we can find a sequence of
continuous functions $\{\phi_{i}^{k}\}_{i=1}^{n_{k}}$ with $n_{k}=2^{k}+1$
such that:

\begin{description}
\item[\rm{(i)}] the diameter of support $\lambda($supp$(\phi_{i}^{k}%
))\leq\frac{2}{2^{k}}$ and $0\leq\phi_{i}^{k}\leq1$;

\item[\rm{(ii)}] $\sum_{i=1}^{n_{k}}\phi_{i}^{k}(t)=1,\text{ for each
}t\in\lbrack0,1]$;

\item[\rm{(iii)}] $\phi_{i}^{k}(t)>0$ for some $t\in\lbrack\frac
{i-1}{2^{k}},\frac{i}{2^{k}})$ but $\phi_{i}^{k}(t)\equiv0$ for $t\geq\frac
{i}{2^{k}},$ for $1\leq i\leq n_{k}+1.$
\end{description}

It is easy to check that
\[
\sum_{i=1}^{n_{k}}I_{[0,\frac{i}{2^{k}}]}\phi_{i}^{k}(\tau)\rightarrow
I_{[0,\tau]}\ \text{in}\ M_{G}^{p}(0,T),\ \ \text{ as }k\rightarrow\infty.
\]
Then it remains to show that $\sum_{i=1}^{n_{k}}I_{[0,\frac{i}{2^{k}}]}%
\phi_{i}^{k}(\tau)\in M_{G}^{p}(0,T)$. A rewriting gives
\begin{align*}
\sum_{i=1}^{n_{k}}I_{[0,\frac{i}{2^{k}}]}\phi_{i}^{k}(\tau)=  &  \sum
_{i=1}^{n_{k}}(\sum_{j=1}^{i}I_{(\frac{j-1}{2^{k}},\frac{j}{2^{k}}]}%
+I_{\{0\}})\phi_{i}^{k}(\tau)\\
=  &  \sum_{j=1}^{n_{k}}\sum_{i=j}^{n_{k}}I_{(\frac{j-1}{2^{k}},\frac{j}%
{2^{k}}]}\phi_{i}^{k}(\tau)+\sum_{i=1}^{n_{k}}I_{\{0\}}\phi_{i}^{k}(\tau)\\
=  &  \sum_{j=1}^{n_{k}}I_{(\frac{j-1}{2^{k}},\frac{j}{2^{k}}]}\sum
_{i=j}^{n_{k}}\phi_{i}^{k}(\tau)+I_{\{0\}}.
\end{align*}
Noting that $1\geq\sum_{i=j}^{n_{k}}\phi_{i}^{k}\geq I_{[\frac{j-1}{2^{k}}%
,1]}$, then
\[
\sum_{i=j}^{n_{k}}\phi_{i}^{k}(\tau)=\sum_{i=j}^{n_{k}}\phi_{i}^{k}(\tau
\wedge\frac{j-1}{2^{n}})I_{[\tau\leq\frac{j-1}{2^{k}}]}+I_{[\tau>\frac
{j-1}{2^{k}}]}\in\mathcal{F}_{\frac{j-1}{2^{k}}}.
\]
Since $\phi^{k}_{i}$ is continuous, thus $\sum_{i=j}^{n_{k}}\phi_{i}^{k}%
(\tau)$ is quasi-continuous. Then by applying Theorem \ref{LG-ch}, we deduce
that $\sum_{i=j}^{n_{k}}\phi_{i}^{k}(\tau)\in L_{G}^{p}(\Omega_{\frac
{j-1}{2^{k}}}) $. This completes the proof.
\end{proof}

\begin{remark}\label{Rm4-11}
\upshape{
	\begin{description}
		\item[\rm{ (i)}]
Similar argument shows that $I_{[0,\tau]}\in H_{G}^{p}(0,T)$ under the same assumptions.
\item[\rm{ (ii)}]
 One of the referees provides an alternative short and novel proof to the above proposition.
Indeed, first from Proposition \ref{Myth3.10}, the random variable $B_\tau$ is quasi-continuous. Then by a similar analysis as in the proof of Corollary \ref{Myth4-10}, we have $$
\hat{\mathbb{E}}[|B_\tau|^{p+2}]=\hat{\mathbb{E}}[|\hat{\mathbb{E}}_{\tau}[B_T]|^{p+2}]\leq \hat{\mathbb{E}}[|B_T|^{p+2}]<\infty.
$$ Thus  $B_\tau\in L_G^{p+1}(\Omega_{T})$. Now we can apply the $G$-martingale representation theorem (see \cite{STZ,Song1}) to obtain a process
 $
 h\in H_G^p(0,T)
 $
 such that
 $$
 B_\tau=\int_0^Th_sdB_s.
 $$
 This implies that $||I_{[0,\tau]}-h||_{H_G^p(0,T)}=0$, and thus $I_{[0,\tau]}\in H_G^p(0,T).$ In particular, $I_{[0,\tau]}\in H_G^2(0,T)=M_G^2(0,T)$. Combining this with the characterization theorem of $M_G^p(0,T)$ (see \cite{HWZ}), we obtain that $I_{[0,\tau]}\in M_G^p(0,T)$.

But we can still keep our original proof because it is more direct and constructive, and since it does not rely on the $G$-martingale representation theorem which is from the structure of the probability family for $G$-expectation, it is applicable to the more general case that the $G$-expectation probability family is replaced by an arbitrary given family of probability measures on $\Omega$.
\end{description}
}
\end{remark}

\begin{remark}
\upshape{Let $\tau\leq T$ be a stopping time and $\eta\in H_{G}^{p}(0,T)$. From \cite{LP}, we have \[
	\int_0^{\tau}\eta_sdB^i_s=\int_0^{T}\eta_sI_{[0,\tau]}(s)dB^i_s.
	\]
If $\tau$ is quasi-continuous, then by the above
Proposition \ref{Myth3.6} (see also Remark \ref{Rm4-11} (i)),  we derive that
$\eta I_{[0,\tau]}\in H_{G}^{p}(0,T)$. Such kind of conclusions may be
useful in the localization argument for the stochastic integrals.}
\end{remark}

\section{Examples and counterexamples}

We first present some examples of nonlinear semimartingales $Y$ satisfying the assumptions in  Theorem \ref{Myth3.5}

\begin{example}
\label{Myth5.1} \upshape{
		\begin{description}
			\item[\rm{ (i)}] Let $\mathcal{{P}}$ be
			the weakly compact family of probability measures corresponding to
			$G$-expectation, under which canonical process $B$ is a $G$-Brownian motion.
			Assume that $B$ satisfies $\frac{d\langle B\rangle_{t}}{dt}\geq
			\underline{\sigma}^{2}I_{k\times k}$ for some $\underline{\sigma}^{2}>0.$ Then
			$B$ is quasi-continuous and satisfies the assumption $(H^{\prime})$.
			
			More generally, let $Q$ be the open set we concern. We take $Y$ as the
			solution of a $d$-dimensional SDEs driven by $G$-Brownian motion $B$:
			\[
			X_{t}^{x}=x+\int_0^tb(s,X_{s}^{x})ds+\sum_{i,j=1}^{k}\int_0^th_{ij}(s,X_{s}^{x})d\langle
			B^{i},B^{j}\rangle_{s}+\sum_{j=1}^{k}\int_0^t\sigma_{j}(s,X_{s}^{x})dB_{s}^{j}%
			,\ \ \ t\geq0,\label{SDE}%
			\]
			where $x\in\mathbb{R}^{d}$, $b(t,x),h_{ij}(t,x),\sigma_{j}(t,x):[0,T]\times
			\mathbb{R}^{d}\rightarrow\mathbb{R}^{d}$ are deterministic functions
			continuous in $t$ and Lipschitz in $x,$ with coefficient $L.$ We also assume
			that $\sigma:=(\sigma_{1}\cdots,\sigma_{k})$ is non-degenerate, i.e., there
			exists a constant $\lambda>0$ such that
			\[
			\lambda I_{d\times d}\leq\sigma(t,y)\sigma(t,y)^{T},\ \ \ \text{for all}%
			\ y\in\overline{Q}.
			\]

			First from Proposition \ref{Myth4.5}, the process $X^{x}$ is
			quasi-continuous. Then we show that the assumption $(H)$
			hold. Indeed, we can fix any $R>0$ and define, for $\omega$ satisfying
			$\tau_{Q}(\omega)<\infty,$  the stopping times
			\[
			\sigma^{\omega}(\omega^{\prime})=\inf\{t\geq\tau_{Q}(\omega):X_{t}^{x}%
			(\omega^{\prime})\in(U(X_{\tau_{Q}(\omega)}^{x}(\omega),R))^{c}\},
			\]
			where $(U(X_{\tau_{Q}(\omega)}^{x}(\omega),R))$ is the open ball with center
			$X_{\tau_{Q}(\omega)}^{x}(\omega)$ and radius $R$. Fix any $\delta>0$, then for $\omega^{\prime
			}\in\Omega^{\omega},$  $\phi
			(\tau_{Q}(\omega)+t,X_{\tau_{Q}(\omega)+t}^{x}(\omega^{\prime}))$, for
			$\phi=b,h_{ij},\sigma_{j}$, are bounded for $t\in[0,\sigma^{\omega}(\omega^{\prime})\wedge\delta]$:
			\begin{align*}
			&  |\phi(\tau_{Q}(\omega)+t,X_{\tau_{Q}(\omega)+t}^{x}(\omega^{\prime}))|\\
			&  \leq|\phi(\tau_{Q}(\omega)+t,X_{\tau_{Q}(\omega)+t}^{x}(\omega^{\prime
			}))-\phi(\tau_{Q}(\omega)+t,0)|+|\phi(\tau_{Q}(\omega)+t,0)|\\
			&  \leq L|X_{\tau_{Q}(\omega)+t}^{x}(\omega^{\prime})|+|\phi(\tau_{Q}%
			(\omega)+t,0)|.
			\end{align*}
			Now by the non-degeneracy assumption on $B$ and $\sigma$, it is easy to check that the assumption
			$(H)$ is satisfied.
			
			It worthy pointing out that whereas $(H^{\prime})$ in Remark \ref{Myrem3.1} may not hold. This
			case is one of the main motivation for our general condition $(H)$. For the above
			$G$-SDE, if moreover $b,h_{ij},\sigma_{j}$ are bounded (globally on $Q$), then
			the stronger assumption $(H^{\prime})$  is also satisfied.
			
			\item[\rm{ (ii)}]  In the $G$-expectation space, we take a $d$-dimensional process
			$Y=M+A,$ where $M$ is a symmetric $G$-martingale and $A$ is a quasi-continuous finite variation process, such that $(H)$ or $(H')$ is satisfied $($In the one-dimensional case, this assumption can be
			weakened, see Remark \ref{Re3.1}$)$.
			
			\item[\rm{ (iii)}] Let $Y=B$ and $\mathcal{{P}}$ be a weakly compact family of
			probability measures such that under each $P\in \mathcal{{P}}$, $Y=M^{P}+A^{P}$
			is a semimartingale satisfying
			\[
			\lambda I_{k\times k}\leq \frac{d\langle M^{P}\rangle_{t}}{dt}\leq \Lambda I_{k\times
				k},\ |\frac{dA_{t}^{P}}{dt}|\leq C\ \ \ {on}\ \overline{Q},\ P\text{-a.s.},\text{
				for some constants }0<\lambda \leq \Lambda,C\geq0\text{,}			\]
			as a case considered in \cite{ETZ}. Then $(H')$ is satisfied and obviously $Y$ is quasi-continuous.
		\end{description}
	}
\end{example}

We then consider several counterexamples which showing that the exit times may
not possess the quasi-continuity if the condition $(H^{\prime})$ is violated. Here we  mainly confine
the discussions to the condition $(H^{\prime})$ for the sake of symbol simplicity, although the condition $(H)$
can also be checked.

The first example concerns
 on the case that the assumption
tr$[d\langle Y\rangle_{t}]>0$, for each $P$, in $(H^{\prime})$ does not hold.
\begin{example}
\label{Myth5.2} \upshape{
\begin{description}
\item[\rm{(i)}] Let $k=1$ and denote $\omega^{x}$ the path with constant value
$x\in \mathbb{R},$ i.e., $\ \omega_{t}^{x}\equiv x$ for each $t\geq0$. We consider the
family  $\mathcal{{P}=}\{P_{x}:x\in\lbrack-1,1]\}$\ of probability measures
such that
\[
P_{x}(\{\omega^{x}\})=1, \ \ \ \text{for each}\ x\in [-1,1].
\]
Take $Q=(-\infty,0)$ and $Y=B$. It is easy to see that $\mathcal{{P}}$ is
weakly compact and $\langle B\rangle^P_{t}\equiv0$ for each $P\in\mathcal{{P}}$.  Note that
\[
({\tau}_{{Q}}\wedge1)(\omega^{x})=0\text{ for }x\in\lbrack0,1],\ \ \text{\ and}\ \
\ ({\tau}_{{Q}}\wedge1)(\omega^{x})=1 \text{ for }x\in\lbrack-1,0).
\]
So $\omega^{0}$ is a discontinuity point of ${\tau}_{{Q}}\wedge1$. Assume
on the contrary that we can find a set $E$  such that
$c(E)\leq\frac{1}{2}$ and ${\tau}_{{Q}}\wedge1$ is continuous on $
E^{c}.$ For each $x\in[-1,1]$, since $c(\{\omega^{x}\})=1,$ so it must hold that $\omega^{x}\in
 E^{c}$. But this contradicts to the
assumption that ${\tau}_{{Q}}  \wedge1$ is continuous on $ E^{c}.$
Therefore, ${\tau}_{{Q}}  \wedge1$ is not quasi-continuous
\item[\rm{(ii)}] Let $k=1$ and $\mathcal{{P}}$ be a weakly compact family of
probability measures, under which $B$ is a one-dimensional $G$-Brownian motion
with $\Gamma=[0,\overline{\sigma}^{2}]$ for some $\overline{\sigma}^{2}>0$.
Assume that under $P_{\sigma}\in\mathcal{P}$, $B$ is a linear Brownian motion
such that$\ \langle B\rangle^{P_{\sigma}}_{t}=\sigma^{2}t,$ for each $\sigma\in
\lbrack0,\overline{\sigma}]$. Take $Q=(-\infty,0)$ and $Y=B$. In this
$G$-Brownian motion case, we need to consider another kind of neighborhood
for  $\omega^{0}$, where $\omega^{0}$ is defined as in (i). Let us denote
\[
A:=\{\omega\in\Omega:\omega_{0}=0,\ (\omega_{t})_{t\geq0}\text{ changes sign
infinitely many times in }[0,\varepsilon],\text{ for each }\varepsilon>0\}.
\]
Then
\[
({\tau}_{{\overline{Q}}}\wedge1)(\omega)=0\text{ for }\omega\in A,\ \ \text{\ and}\  \ \ ({\tau}_{{\overline{Q}}}\wedge1)(\omega^{0})=1.
\]
This means that ${\tau}_{\overline{Q}}\wedge1$ is not continuous at
$\omega^{0}.$
Now we show that ${\tau}_{\overline{Q}}\wedge1$ is not quasi-continuous.
Indeed, for any given $T>0$ and $\varepsilon>0,$ since $\frac{B_{t}}{\sigma}$ is a standard Brownian motion, then
\[
P_{\sigma}(\{\sup_{0\leq t\leq T}|B_{t}|\leq\varepsilon\})=P_{\sigma}
(\{\sup_{0\leq t\leq T}|\frac{B_{t}}{\sigma}|\leq\frac{\varepsilon}{\sigma
}\})\rightarrow 1,\   \ \ \text{as}\ \ 0<\sigma\downarrow0.
\]
Thus,
\[
P_{\sigma}(\{\omega\in\Omega:\rho(\omega,\omega^{0})\leq\varepsilon
\})\rightarrow 1,\   \ \ \text{as}\ \ 0<\sigma\downarrow0.
\]
Therefore, by the path property of linear Brownian motion $($see Problem
2.7.18 of \cite{KS}$)$,
\[
P_{\sigma}(\{\omega\in A:\rho(\omega,\omega^{0})\leq\varepsilon\})\rightarrow
1,\ \ \ \text{as}\ \ 0<\sigma\downarrow0.
\]
This implies
\begin{equation}
\label{Myeq5.5}c(A_{\varepsilon})=1,\text{ for each }\varepsilon>0,\ \text{
where } A_{\varepsilon}:=\{\omega\in A:\rho(\omega,\omega^{0})\leq
\varepsilon\}.
\end{equation}
Assume on the contrary that we can find a set $E$ such that $c(E)\leq\frac
{1}{2}$ and ${\tau}_{\overline{Q}}\wedge1$ is continuous on $
E^{c}.$ Since $P_0(\{\omega^0\})=1$, so $c(\{\omega^0\})=1$, and thus $\omega^{0}\in E^c$. Note that $\omega^{0}\in E^c$ is a limit point of $A\cap E^{c},$ since if
not, there exists some $\varepsilon>0$ such that $A_{\varepsilon}\subset E,$ which is impossible by
equality (\ref{Myeq5.5}). Thus we have reached a contradiction. So
${\tau}_{\overline{Q}} \wedge1$ is not quasi-continuous.
\end{description}
}
\end{example}

Now we give an example in which $d\langle Y\rangle_{t}\geq\varepsilon$
tr$[d\langle Y\rangle_{t}]I_{d \times d}$ for some $\varepsilon>0$, for each
$P$, in $(H^{\prime})$ is not met.

\begin{example}
	\upshape{
		Let $k=2$ and $\mathcal{{P}}$ be the weakly compact family of probability
		measures such that $B$ is a two-dimensional $G$-Brownian motion with
		\[
		\Gamma=\left\{  \left[
		\begin{tabular}
		[c]{ll}$\alpha$ & $0$\\
		$0$ & $1-\alpha$
		\end{tabular}
		\right]  :0\leq\alpha\leq1\right\} .
		\]
		Then tr$[\langle B\rangle^P_{t}]=t$, for each $P\in
		\mathcal{P}$.
		Assume that under $P_{\alpha}\in\mathcal{P}$, $B$ is a linear Brownian motion
		with $\langle B\rangle^{P_{\alpha}}_{t}=t\left[
		\begin{tabular}
		[c]{ll}$\alpha$ & $0$\\
		$0$ & $1-\alpha$\end{tabular}
		\right]  ,$ for each $0\leq\alpha\leq1$. Let us take $Q=(-\infty,\infty
		)\times(0,1)\ $and $Y=B$. We identify $\omega=(\omega^{1},\omega^{2})$, where $\omega^{j},j=1,2$ are the corresponding scalar components. In this example, we need to consider the following
		set of points of discontinuity:
		\[
		\Omega_{0}:=\{\omega=(\omega^{1},\omega^{2})\in\Omega:\omega_{t}^{2}\equiv0,\ t\geq0\}.
		\]
		We define
		\[
		A:=\{\omega\in\Omega:\omega_{0}^{2}=0,\ (\omega_{t}^{2})_{t\geq0}\text{
			changes sign infinitely many times in }[0,\varepsilon],\text{ for each
		}\varepsilon>0\}.
		\]
		Since $(\omega_{t}^{2})_{t\geq0}$ is a linear Brownian motion under
		$P_{\alpha}$, then
		 $P_\alpha(A)=1$ for $\alpha< 1$. It is easy to see that
		\[
		({\tau}_{{\overline{Q}}}\wedge1)(\omega)=0\text{ for }\omega\in A, \ \ \ \text{and}\ \
		\ ({\tau}_{{\overline{Q}}}\wedge1)(\omega)=1\text{ for }\omega\in\Omega_{0},
		\]
		which means that each $\omega\in\Omega_{0}$ is a discontinuity point of
		${\tau}_{{\overline{Q}}}\wedge1.$
		Assume that we can find a set $E$ such that  $c(E)\leq\frac{1}{2}$ and ${\tau
		}_{\overline{Q}}\wedge1$ is continuous on  $E^{c}.$ If
		$\omega\in\Omega_{0}$ is a limit point  $A\cap E^{c}$, since $E^c$ is closed,  we will have $\omega\in E^c$, which leads to the discontinuity of ${\tau
		}_{\overline{Q}}\wedge1$ on $E^c$. So any
	$\omega\in\Omega_{0}$ should not be a limit point  of $A\cap E^{c}$, and thus there
		exists an open set $O\subset\Omega$ such that  $O\supset\Omega_{0}$ and
		$O\cap (A\cap E^c)=\emptyset$.
		 Now we claim  that $c(O\cap A)=1$. Indeed, since $P_{\alpha}$ converges
		to $P_{1}$ weakly, as $\alpha\rightarrow1$, then
		\[
		\liminf_{1>\alpha\rightarrow1}P_{\alpha}(O\cap
		A)=\liminf_{1>\alpha\rightarrow1}P_{\alpha}(O)\geq P_{1}(O)=P_{1}(\Omega_{0})=1.
		\]
		This implies
		\[
		c(O\cap A)=1,
		\]
		which is a contradiction since $O\cap A\subset E$. Therefore, ${\tau}_{{Q}}\wedge1$ is not quasi-continuous.}
\end{example}
\bigskip

\noindent\textbf{Acknowledgement}:
The author would like to thank Shige Peng and
Yongsheng Song for their helpful discussions.
The author is also very grateful to the anonymous referees for their very careful reading and many
valuable suggestions.

\end{document}